\documentclass[reqno, 11pt]{amsart}

\usepackage{amsmath}
\usepackage{amsthm}
\usepackage{amssymb}
\usepackage{dsfont}
\usepackage{graphicx}
\usepackage{caption}
\usepackage{subcaption}
\usepackage{color}
\usepackage[english]{babel}
\usepackage{layout}
\usepackage{float}

\theoremstyle{plain}
\begingroup
\newtheorem{theorem}{Theorem}[section]
\newtheorem{lemma}[theorem]{Lemma}
\newtheorem{proposition}[theorem]{Proposition}
\newtheorem{corollary}[theorem]{Corollary}
\endgroup

\theoremstyle{definition}
\begingroup

\newtheorem{remark}[theorem]{Remark}

\endgroup

\theoremstyle{remark}

\mathchardef\emptyset="001F

\numberwithin{equation}{section}

\newcommand{\e}{\varepsilon}

\newcommand{\dist}{{\rm{dist}}}

\font\tenmsb=msbm10
\font\sevenmsb=msbm7
\font\fivemsb=msbm5
\newfam\msbfam
\textfont\msbfam=\tenmsb
\scriptfont\msbfam=\sevenmsb
\scriptscriptfont\msbfam=\fivemsb

\def\R{\mathbb{R}}

\def\w{\omega}

\def\a{\alpha}

\begin{document}

\author{Marco Cicalese}
\address[Marco Cicalese]{Zentrum Mathematik - M7, Technische Universit\"at M\"unchen, Boltzmannstrasse 3, 85747 Garching, Germany}
\email{cicalese@ma.tum.de}

\author{Matthias Ruf}
\address[Matthias Ruf]{Zentrum Mathematik - M7, Technische Universit\"at M\"unchen, Boltzmannstrasse 3, 85747 Garching, Germany}
\email{mruf@ma.tum.de}

\author{Francesco Solombrino}
\address[Francesco Solombrino]{Zentrum Mathematik - M7, Technische Universit\"at M\"unchen, Boltzmannstrasse 3, 85747 Garching, Germany}
\email{francesco.solombrino@ma.tum.de}

\title{On global and local minimizers of prestrained thin elastic rods}

\begin{abstract}
We study the stable configurations of a thin three-dimensional weakly prestrained rod subject to a terminal load as the thickness of the section vanishes. By $\Gamma$-convergence we derive a one-dimensional limit theory and show that isolated local minimizers of the limit model can be approached by local minimizers of the three-dimensional model. In the case of isotropic materials and for two-layers prestrained three-dimensional models the limit energy  further simplifies to that of a Kirchhoff rod-model of an intrinsically curved beam.  In this case we study the limit theory and investigate global and/or local stability of straight and helical configurations. Through some simple simulations we finally compare our results with real experiments. 
\end{abstract}

\maketitle

\tableofcontents
\section{Introduction}

Subject to proper boundary conditions or body forces, elastic materials undergo deformations strongly depending on their shapes. For thin three-dimensional bodies, low-energy deformations depend on the shape to such an extent that elastic theories for lower dimensional bodies like rods, plates or shells have been developed (see for instance \cite{Ciarlet, Love}). The rigorous ansatz-free variational derivation of such theories from three-dimensional nonlinear elasticity is a very well established research field. In few words the problem one aims to solve is the derivation of the elastic energy of a thin object by taking a suitable limit of the elastic energy of the three-dimensional body as its thickness vanishes. Fundamental issues, like the convergence of global minimizers and of critical points in the limit process have been discussed in the literature and a hierarchy of theories depending on the scaling of the energies (or the forces) in terms of the thickness has been derived since the pioneering papers by Le Dret and Raoult \cite{LR1} and by Friesecke, James and M\"uller \cite{FJM, FJM1} by many authors \cite{ABP, MM1, MM2, MS, Scardia1, Scardia2} under different modelling assumptions. \\

More recently the problem above has gained increasing attention in the case of prestrained bodies. A number of results have appeared in the case of $3$-d to $2$-d dimension reduction in \cite{BMS, FHMP, LMP1, LMP2, S} and many interesting questions have been raised. 
On one hand the above problem has been left undiscussed in the case of $3$-d to $1$-d dimension reduction (see \cite{ADS, ADSK} for a similar problem in the theory of nematic elastomers where the dimension reduction is performed in two subsequent steps $3$-d to $2$-d and $2$-d to $1$-d), on the other hand recent experiments in \cite{Plos} suggest to consider it from a rigorous mathematical point of view. 
In few words in \cite{Plos} the authors take two long strips of elastomer of the same initial width, but unequal length. The short strip is stretched uniaxially to be equal in
length to the longer one. The initial heights are chosen so that after stretching the bi-strip system has a  rectangular cross section. The two strips are then glued together side-by-side along their length. The bi-strips appear flat and the initially shorter strip is under a uniaxial prestrain. As a last step of the experiment, the external force needed to stretch the ends of the bi-strip is gradually released so that the initially flat bi-strip starts to bend and twist out of plane. It may evolve towards either a helical or hemihelical shape (more complex structure in which helices with different chiralities seem to periodically alternate), depending on the cross-sectional aspect ratio. In particular, a big enough aspect ratio favors the formation of a helix, whereas a small aspect ratio favors that of a hemihelix. 

The analysis in \cite{Plos} is simplified first assuming that the system is one-dimensional, so that a Kirchhoff-rod approximation is used, and then analyzing stability of configurations close to the straight rod by matching asymptotics in a restricted class of competitors. On one hand the results appear to be mathematically unsatisfying, on the other hand a rigorous derivation of the complete observed behavior seems to be quite challenging. In this paper we aim at partially bridging this gap by first providing the 1-d limit theory by Gamma-convergence and then proving approximation of isolated local minima of the 1-d energy by local minima of the 3-d energy. In our opinion the second aspect constitutes the real novelty of this paper together with some second order criteria for strict local minimality of 1-d configurations. Second order criteria for Kirchhoff-rod energies already appeared in the literature \cite{M, MPG, MR} but lacking any rigorous relation to the 3-d problem (see below for other differences between the two approaches). Furthermore we make use of our criteria to test the stability of the 1-d straight and helical configurations in terms of critical forces. By means of numerical experiments we make some rough comparison between our analysis and the experimental results in \cite{Plos}.\\
 
We now detail our setting. We consider a prestrained thin $3$-dimensional stripe subject to terminal loads. Given a small parameter $h>0$, a stripe of thickness $h$, mid-fiber $(0,L)$ and cross section $S\subset\R^2$ is denoted by $\Omega_h:=(0,L)\times hS$, with $S$ being a symmetric, smooth enough open set of unit area. To model the prestrain we introduce a measurable matrix-valued field $\overline {A}_h:\Omega_h\to \mathbb{M}^{3\times 3}$ (note that here we do not assume any further regularity on $A_{h}$ as for instance in \cite{LMP1, LMP2} as we are interested in discontinuous prestrain to compare our results with the experiments in \cite{Plos}). Denoting by $\Gamma_h\subset\{0,L\}\times hS$ a portion of the boundary of the beam with positive two-dimensional measure, we consider a force field $\overline{f}_h:\Gamma_h\to\R^3$ acting on the beam. Further denoting by $W: \mathbb{M}^{3\times 3}\to [0,+\infty]$ the strain energy density we introduce the total energy of the system stored by a deformation $u:\Omega_h\to \R^3$ as
\begin{equation*}
\overline{E}_h(u)=\int_{\Omega_h}W(\nabla u(x)\overline {A}_h(x))\,\mathrm{d}x-\int_{\Gamma_h}\langle\overline{f}_h(x),u(x)\rangle\,\mathrm{d}\mathcal{H}^{2}.
\end{equation*}
As usual in the analysis of thin elastic objects, we perform a change of variables to rewrite the energy on a fixed domain. We set $\Omega=\Omega_1,\,\Gamma=\Gamma_1$ and define the new variable $v:\Omega\to\R^3$, the rescaled prestrain $A_h:\Omega\to\mathbb{M}^{3\times 3}$ and force $f_h:\Gamma\to\R^3$ as
\begin{equation*}
v(x)=u(x_1,hx_2,hx_3),\; A_h(x)=\overline {A}_h(x_1,hx_2,hx_3),\; f_h(x)=\overline{f}_h(x_1,hx_2,hx_3).
\end{equation*}
In terms of the rescaled gradient $\nabla_hv=\partial_1 v\otimes e_1+\frac{1}{h}(\partial_2 v\otimes e_2+\partial_3\otimes e_3)$, the energy can be rewritten as $\overline{E}_h(u)=h^{2}E_h(u)$, where
\begin{equation*}
E_h(u)=
\int_{\Omega}W(\nabla_hv(x) A_h(x))\,\mathrm{d}x-\int_{\Gamma}\langle f_h(x),u(x)\rangle\,\mathrm{d}\mathcal{H}^2.
\end{equation*}
We will be interested in the case where 
\begin{equation*}
\frac{f_h}{h^2}\rightharpoonup f \quad\text{in }L^2(\Gamma)
\end{equation*}
suggesting that a meaningful scaling for $E_{h}$ is $E_{h}/h^{2}$. Energies $E_{h}$ of order $h^{2}$ correspond to bending flexures and torsions and lead to a rod theory. To prevent the energy functional to be unbounded from below due to the invariance under translations we will further assume $v(0,x_2,x_3)=R_{0}(0,hx_{2},hx_{3})$ for some $R_{0}\in SO(3)$. Under standard frame indifference, non-degeneracy and regularity assumptions on $W$ (see the next section for the precise set of assumptions) and assuming the rescaled prestrain to be a small perturbation of the identity in the sense of \eqref{weakprestrain}, we consider the $\Gamma$-limit as $h\to 0$ of $E_{h}/h^{2}$ under several boundary conditions comprising clamped and weak clamped beams (at both ends the beam is free to rotate on a plane orthogonal to a fixed direction as in \cite{Plos} or \cite{tendril}). We notice that such kind of boundary conditions may involve some additional technical difficulties (see for instance Section \ref{convisolated}) and, to the best of our knowledge, their role in the variational analysis of stability via $\Gamma$-convergence has not been studied. The quadratic scaling has been first studied by Mora and M\"uller in \cite{MM1, MM2} in the case $A_{h}$ is the identity matrix (see also \cite{Scardia1, Scardia2}). Under the assumptions above, for such energies one can prove a $\Gamma$-convergence result 
and characterize the limit energy density following the same steps as in \cite{MM1}. Using the notation $R(x_1)=(\partial_1 v(x_1)|d_2(x_1)|d_3(x_1))$, Theorem \ref{1dlimit} shows that the limit energy $E_0:W^{1,2}(\Omega,\R^3)\times (L^2(\Omega,\R^3))^2\to [0,+\infty]$ is finite on the space 
\begin{align*}
\mathcal{A}:=\Bigl\{(v,d_2,d_3)\in W^{2,2}(\Omega,\R^3)\times (W^{1,2}(\Omega,\R^3))^2\text{ depend only on }&x_1,\,v(0)=0,\\  
 &(\partial_1v|d_2|d_3)\in SO(3) \text{ a.e.}\Bigr\}.
\end{align*}
where it has the form
\begin{equation}\label{intro-deflimit}
E_0(R):=
\frac{1}{2}\int_0^LQ_2\left(x_1,R^T(x_1)R^{\prime}(x_1)\right)\,\mathrm{d}x_1-\int_{\Gamma}\langle f(x),v(x)\rangle\,\mathrm{d}\mathcal{H}^2 
\end{equation}
with $Q_{2}(x_{1},A)$ being a quadratic function of the entries of $A$ which is obtained by a minimization problem involving the quadratic form of linear elasticity (see Proposition \ref{ondensity}).\\
The formula for $Q_{2}$ can be simplified when the energy density $W$ is isotropic. Having in mind the experiments in \cite{Plos}, we assume additional symmetry of $S$, covering the case of a rectangular cross section as in the experimental set-up. Furthermore we consider a prestrain matrix which describes locally an incompressible deformation ($\det \overline A_{h}=1$) and has a two-layer structure of the form
\begin{equation}\label{intro-twolayer}
\overline A_h(x):=\begin{cases}
{\rm{diag}}(1+h\chi,\frac{1}{\sqrt{1+h\chi}},\frac{1}{\sqrt{1+h\chi}}) &\mbox{if $x_3>0$,}\\
I &\mbox{if $x_3<0$,}
\end{cases}
\end{equation}
where $\chi>0$ is the effective strength of the stretching. Under this assumptions, in Proposition \ref{curvaturelimit} we prove that  
there exist positive coefficients $c_{12}, c_{13}, c_{23}$ such that, up to an additive constant, the density $Q_2$ in \eqref{intro-deflimit} is given by
\begin{equation}\label{intro-abstractdensity}
Q_2(A)=c_{12}a_{12}^2+c_{13}(a_{13}-k)^2+c_{23}a_{23}^2,
\end{equation}
where
\begin{equation}\label{intro-curvature-prestrain}
k=\chi \frac{\int_{S^{+}}x_3\,\mathrm{d}\mathcal{H}^2}{\int_Sx_3^2\,\mathrm{d}\mathcal{H}^2}
\end{equation}
is the intrinsic curvature. This implies in particular that, without external forces, a uniformly curved beam is a global minimizer (and actually, up to rotations, the unique global minimizer) of the limit functional. As a result, by the general theory of $\Gamma$-convergence, for $h$ small enough, global minimizers of the energy $E_{h}$ with zero external forces are close to a curved beam (see Proposition \ref{curvaturelimit} and subsequent remarks).\\ 

The analysis of the convergence of global minimizers has been complemented in literature with results on the convergence of critical points of the $3$-d functional to critical points of the $1$-d limit energy. For models without prestrain we may mention the papers \cite{MM2}, where more regular energy densities are taken into account, or \cite{MS}, for a stronger scaling. Since instead we are also interested in validating the usage of stable states of the one-dimensional limit as approximations of stable states of the three-dimensional energy, we further analyse the problem from the reverse point of view. We namely show that, for a given (strict) local minimizer of the 1-d functional in the $W^{1,2}$-topology, it exists a sequence of local minimizers of the 3-d one converging to it. The choice of the topology, as we also discuss in Section \ref{convisolated}, complies well with simple criteria for showing local minimality, which we also provide.

Going into detail, we first prove that, thanks to the quadratic structure of the limit functional $E_{0}$, if $R$ is a strict local minimizer with respect to the $W^{1,2}$-topology, then it is also a strict local minimizer with respect to the $L^{2}$-topology where it is possible to apply the results in \cite{KS} (see also \cite{BrLoc}) about approximation of strict local minima by $\Gamma$-convergence. However the advantage of working in $W^{1,2}$ (instead of directly considering $L^2$) is that it is possible to characterize local minima of $E_{0}$ by means of its second variation. To this end in Proposition \ref{variations} we compute first and second variations of the functional $E_{0}$ and eventually prove in Theorem \ref{mainlocal} that if the second variation of $E_{0}$ at a critical point $R$ is a positive-definite quadratic form, then $R$ is a strict local minimizer in $L^{2}$. We remark that we do not use Euler angles to rewrite $E_{0}$ in contrast to \cite{M,MPG,MR}. On one hand this makes our problem mathematically more complicated since the domain of our energy functional is not a linear space, on the other we gain in generality since we have to pose no a-priori restrictions to the configurations in order to avoid polar singularities. It is worth noticing that the proof of our local minimality criterion is simplified by the possibility of extending the limit energy functional $E_0$ in a neighborhood of $SO(3)$ to a functional $\tilde E_{0}$ twice continuously Fr\'{e}chet-differentiable in $W^{1,2}$. Such a strong property is a consequence of two facts: the quadratic dependence of $\tilde E_{0}(M)$ on $M'$ and its smooth dependence on $M$. The first is due to the quadratic structure of the limit energy while the second to the dimensionality of the problem: being $\tilde E_{0}$ defined on one-dimensional $W^{1,2}$ functions, Sobolev embeddings give $L^\infty$-compactness for $M$ which turns out to provide a smooth dependence of $\tilde E_{0}$ on $M$ (see Lemma \ref{frechetdiff}).\\ 

In the last section of the paper we investigate the local stability of straight and helical configurations under terminal loads. A similar problem, with the already discussed restrictions and in absence of prestrain, appeared in \cite{M, MPG, MR}. In Theorem \ref{idstable} we compute a critical value (explicitly depending on the boundary conditions) $f^{crit}$ of the terminal load $f$ such that, for $f>f^{crit}$ the straight configuration is a  $L^{2}$ local minimizer, while for $f<f^{crit}$ it is not. This result also enable us to prove that the straight configuration is the unique global minimizer of $E_{0}$ for sufficiently large loads $f\geq f_{g}>f^{crit}$. A stability analysis, based on a conjugate point method, is performed in Section \ref{stabhel} for helical solutions. The results of this section are also exploited to provide numerical evidence of some of the experimentally observed behavior of the physical model in \cite{Plos}. At the critical force, under some condition on the parameters of the problem, helical solutions arbitrarily close at the origin to the straight configuration emerge as a branch of local minima of the energy $E_{0}$ with respect to their own boundary conditions. If instead the same boundary conditions are kept, both in the clamped-clamped and in the weak-clamped case, Theorem \ref{bifurcation} shows that for $f=f^{crit}$ there exist a branch of critical points bifurcating from the identity. 

The last part of the paper is devoted to simple numerical experiments in which we show that two factors can influence the stability of helical solutions, namely the aspect ratio of the rectangular cross section and the intrinsic curvature of the rod through the prestrain parameter $\chi$ in \eqref{intro-twolayer}.

\section{Notation and preliminaries on the model}
We denote by $\{e_{1},e_{2},e_{3}\}$ the standard basis in $\R^{3}$, by $\mathbb{M}^{3\times 3}$ the set of all real-valued $3\times3$ matrices and by $I$ the identity matrix. Given a matrix $M$ we denote by $M^{T}$ its adjoint matrix (this convention will be also used for row and column vectors). The entries of $M$ will be denoted by $m_{ij}$. We further set ${\rm sym}(M)=\frac{1}{2}(M+M^T)$. We let $SO(3)$ be the set of all rotations, while $\mathbb{M}^{3\times 3}_{sym}$ and $\mathbb{M}^{3\times 3}_{skew}$ denote the symmetric and skew-symmetric matrices, respectively. Given $A\in\mathbb{M}^{3\times 3}_{skew}$ we define $\omega_{A}$ as the unique vector such that $Av=\omega_{A}\times v$ for all $v\in\R^3$. \\
All euclidean spaces will be endowed with the canonical euclidean norm. 
Given two vectors $v,w\in \R^3$ we denote by $\langle v,w\rangle$ the scalar product and by $v\otimes w$ the dyadic product. Moreover, we define ${\rm diag}(v)\in\mathbb{M}^{3\times 3}$ as the diagonal matrix with entries ${\rm{diag}}(v)_{ij}=v_i\delta_{ij}$. All over the paper $C$ denotes a generic constant that may change from line to line.
The derivative of one-dimensional absolutely continuous functions will be denoted by $^{\prime}$. 

\subsection{Mathematical modeling of prestrained elastic materials}
We will consider a prestrained thin $3$-dimensional stripe. Given a small parameter $h>0$, a stripe of thickness $h$, mid-fiber $(0,L)$ and cross section $S\subset\R^2$ is denoted by $\Omega_h:=(0,L)\times hS$. On $S$ we will assume that it is a bounded open connected set having unitary area and Lipschitz boundary. We moreover assume that $S$ satisfies the following symmetry properties:  

\begin{equation}\label{sym}
\int_Sx_2x_3\,\mathrm{d}x_2\mathrm{d}x_3=\int_Sx_2\,\mathrm{d}x_2\mathrm{d}x_3=\int_Sx_3\,\mathrm{d}x_2\mathrm{d}x_3=0.
\end{equation}

The prestrain is defined as a measurable matrix-valued field $\overline {A}_h:\Omega_h\to \mathbb{M}^{3\times 3}$.
We require the orientation preserving condition
\begin{equation}\label{invertible}
\det(\overline {A}_h(x))>0\quad \text{a.e. in }\Omega_h.
\end{equation}
We consider a hyperelastic material and assume a multiplicative decomposition for the strain (see \cite{LMP1}). Denoting by $W: \mathbb{M}^{3\times 3}\to [0,+\infty]$ the strain energy density, the stored energy of a deformation $u:\Omega_h\to \R^3$ is expressed by 
\begin{equation*}
E_h(u)=\int_{\Omega_h}W(\nabla u(x) \overline {A}_h(x))\,\mathrm{d}x.
\end{equation*}
Throughout the paper we make the following standard assumptions on the density $W$:
\begin{itemize}
\item[(i)] $W(RF)=W(F)\quad\forall R\in SO(3)\quad$ (frame indifference),
\item[(ii)] $W(F)\geq c\,\dist^2(F,SO(3))$ and $W(I)=0\quad$ (non-degeneracy),
\item[(iii)] $W$ is $C^2$ in a neighborhood $U$ of $SO(3)$.
\end{itemize}
In addition to the stored energy we want to consider an external boundary force of the type described below. Given a portion $\Gamma_h\subset\{0,L\}\times hS$ with $\mathcal{H}^2(\Gamma_h)>0$ and a force field $\overline{f}_h:\Gamma_h\to\R^3$ we define the total energy as
\begin{equation*}
\overline{E}_h(u)=\int_{\Omega_h}W(\nabla u(x)\overline {A}_h(x))\,\mathrm{d}x-\int_{\Gamma_h}\langle\overline{f}_h(x),u(x)\rangle\,\mathrm{d}\mathcal{H}^{2}.
\end{equation*}
As it is customary when dealing with the variational analysis of thin objects, we perform a change of variables to rewrite the energy on a fixed domain. Setting $\Omega=\Omega_1,\,\Gamma=\Gamma_1$, we define the new variable $v:\Omega\to\R^3$ and the rescaled prestrain $A_h:\Omega\to\mathbb{M}^{3\times 3}$ and force $f_h:\Gamma\to\R^3$ as
\begin{equation*}
v(x)=u(x_1,hx_2,hx_3),\; A_h(x)=\overline {A}_h(x_1,hx_2,hx_3),\; f_h(x)=\overline{f}_h(x_1,hx_2,hx_3).
\end{equation*}
Introducing the rescaled gradient $\nabla_hv=\partial_1 v\otimes e_1+\frac{1}{h}(\partial_2 v\otimes e_2+\partial_3v\otimes e_3)$, the energy takes the form $\overline{E}_h(u)=h^{2}E_h(u)$, where
\begin{equation}\label{defenergy}
E_h(u)=
\int_{\Omega}W(\nabla_hv(x) A_h(x))\,\mathrm{d}x-\int_{\Gamma}\langle f_h(x),u(x)\rangle\,\mathrm{d}\mathcal{H}^2.
\end{equation}
We will be interested in the case where 
\begin{equation}\label{forcescale}
\frac{f_h}{h^2}\rightharpoonup f \quad\text{in }L^2(\Gamma)
\end{equation}
suggesting that a meaningful scaling for $E_{h}$ is $E_{h}/h^{2}$.

As a matter of fact, for a generic force term $E_{h}/h^{2}$ might negatively diverge since it does not control translations. To rule this out we supplement the problem with the following boundary condition allowing at one end only rotations: 
\begin{equation}\label{weakbc}
v(0,x_2,x_3)=R_0\begin{pmatrix}
0 \\ hx_2\\hx_3
\end{pmatrix}\quad\text{for some }R_0\in SO(3).
\end{equation}
Stronger boundary conditions of Dirichlet type can be also considered and will be discussed in Section \ref{sect-boundary}, namely
\begin{itemize}
\item[]{\bf Clamped}: 
\begin{equation}\label{clamp}
v(0,x_2,x_3)=R_{0}\begin{pmatrix} 0\\ hx_2\\hx_3
\end{pmatrix}
\quad\hbox{ for a fixed } R_{0}\in SO(3)
\end{equation} 
\item[]{\bf Clamped-Clamped}: in addition to \eqref{clamp} it also holds 
\begin{equation}\label{clamp-clamp}
v(L,x_2,x_3)-\int_Sv(L,x_2,x_3)\,\mathrm{d}\mathcal{H}^2=R_{L}\begin{pmatrix} 0\\ hx_2\\hx_3
\end{pmatrix}
\quad \text{for a fixed }R_{L}\in SO(3).
\end{equation}
\end{itemize}
We will see in Corollary \ref{convbc} that the above constraints will imply Dirichlet boundary conditions for the limit strains.\\
A weaker constraint, which we call {\bf weak clamping} (at one or both ends) has been also considered in literature (see for instance \cite{Plos}). In this case one leaves the ends free to rotate in the plane orthogonal to a fixed direction. Assuming for simplicity this direction to be $e_{1}$ and considering the more general case of a weak clamping at both ends, it amounts to requiring 

\begin{align}\label{rigid}
\nonumber&v(0,x_2,x_3)=M_{0}\begin{pmatrix} 0\\ hx_2\\hx_3
\end{pmatrix}
\hbox{ for some } M_{0}\in SO(3) \text{ satisfying } M_{0}e_{1}=e_{1}\\
&v(L,x_2,x_3)-\int_Sv(L,x_2,x_3)\,\mathrm{d}\mathcal{H}^2=M_{L}\begin{pmatrix} 0\\ hx_2\\hx_3
\end{pmatrix}
\hbox{ for some } M_{L}\in SO(3) \text{ satisfying } M_{L}e_{1}=e_{1}.
\end{align}
The analysis of the energy $E_{h}$ will entail bending and torsion effects as a result of the $h^2$ scaling. Accordingly, in order to obtain a proper limiting theory as the thickness $h$ tends to zero, we assume throughout this paper that the rescaled prestrain is a (not necessarily smooth) perturbation of the identity in the following sense:
\begin{equation}\label{weakprestrain}
\|A_h-I\|_{\infty}\leq C_0\,h,\quad \frac{1}{h}(A_h-I)\to \overline{A}\quad\text{a.e.}
\end{equation}
for some $\overline{A}\in L^{\infty}(\Omega,\mathbb{M}^{3\times 3})$. Note that these assumptions already imply (\ref{invertible}) if $h>0$ is small enough.

Now we can introduce the precise variational framework for our model. We consider admissible deformations of the class
\begin{equation}\label{admissible}
{\mathcal A}_{SO(3)}(\Omega):=\{v\in W^{1,2}(\Omega,\R^3):\;v \text{ fulfills }(\ref{weakbc}) \text{ in the sense of traces.}\}
\end{equation}
and define (with a slight abuse of notation) the energies $E_{h}:W^{1,2}(\Omega,\R^3)\to(-\infty,+\infty]$ as
\begin{equation}\label{totenergy}
E_{h}(v):=
\begin{cases}
\int_{\Omega}W(\nabla_hv(x) A_h(x))\,\mathrm{d}x-\int_{\Gamma}\langle f_h(x),v(x)\rangle\,\mathrm{d}\mathcal{H}^2 &\mbox{if $v\in {\mathcal A}_{SO(3)}(\Omega)$,}\\
+\infty &\mbox{otherwise.}
\end{cases}
\end{equation}

\section{Convergence of global minimizers}
In this Section we study the global minimizers of our model problem analyzing the $\Gamma$-convergence of the sequence of energies $E_{h}$ as $h\to 0$. The analysis will be performed for general strain energy densities first and then specialized to the isotropic case. 

\subsection{$\Gamma$-convergence of general prestrained energies}

We start recalling Proposition 4.1 in \cite{MM2}, which is based on the geometric rigidity estimate in \cite{FJM} and show strong compactness for sequences with finite energy of order $h^2$. 

\begin{proposition}\label{finercompactness}
	Let $v_h\in W^{1,2}(\Omega,\R^3)$ be such that
	\begin{equation*}
	\int_{\Omega}\mathrm{dist}^2(\nabla_h v_h,SO(3))\,\mathrm{d}x\leq Ch^2.
	\end{equation*}
	Then there exists an associated sequence $R_h\in C^{\infty}((0,L),SO(3))$ such that
	\begin{enumerate}
		\item[(i)] $\|\nabla_hv_h-R_h\|_{L^2(\Omega)}\leq Ch$,
		\item[(ii)] $\|\partial_1R_h\|_{L^2(0,L)}\leq C$,
	\end{enumerate}
	If in addition $v_h$ fulfills $v_h(0,x_2,x_3)=(0,hx_2,hx_3)$, then 
	\begin{enumerate}
		\item[(iii)] $|R_h(0)-I|\leq C\sqrt{h}$
	\end{enumerate}
\end{proposition}
\begin{remark}\label{rem-clamped}
If a clamped boundary condition \eqref{clamp} holds, statement (iii) in the previous proposition reads
	\begin{enumerate}
		\item[(iii')] $|R_h(0)-R_{0}|\leq C\sqrt{h}$.
	\end{enumerate}
This follows immediately by frame indifference if we consider the new sequence $\tilde{v_h}=R_0^Tv_h$ and apply the above proposition.
\end{remark}
Using Proposition \ref{finercompactness} we deduce the analogous compactness property for our model. In the statement we consider the space of admissible limit deformations defined as
\begin{align*}
\mathcal{A}:=\Bigl\{(v,d_2,d_3)\in W^{2,2}(\Omega,\R^3)\times (W^{1,2}(\Omega,\R^3))^2\text{ depend only on }&x_1,\,v(0)=0,\\  
 &(\partial_1v|d_2|d_3)\in SO(3) \text{ a.e.}\Bigr\}.
\end{align*}

\begin{proposition}\label{compact}
Let $v_h\in {\mathcal A}_{SO(3)}(\Omega)$ be such that 
\begin{equation*}
E_h(v_h)\leq C\,h^2,
\end{equation*}
then (up to subsequences) $v_h\to v$ strongly in $W^{1,2}(\Omega)$, $\nabla_hv_h\to (\partial_1 v | d_2 | d_3)$ strongly in $L^2(\Omega)$ for some 
$(v,d_{2},d_{3})\in \mathcal{A}$.
\end{proposition}

\begin{proof}
We first show that 
\begin{equation}\label{distest}
\sup_{h>0}\frac{1}{h^2}\int_{\Omega}\dist^2(\nabla_h v_h,SO(3))\,\mathrm{d}x< +\infty.
\end{equation}
Exploiting Poincar\'e's inequality and (\ref{forcescale}) we get
\begin{equation}\label{forcefinite}
\left|\frac{1}{h^2}\int_{\Gamma}\langle f_{h},v_h\rangle\,\mathrm{d}\mathcal{H}^2\right|\leq \|h^{-2}f_{h}\|_{L^2(\Omega)}\|v_h\|_{L^2(\Omega)}\leq C(\|\nabla v_h\|_{L^2(\Omega)}+h).
\end{equation}
Since $|\nabla_{h_j}v_j|^2\geq|\nabla v_j|^2$ one can use the elementary inequality $|A-B|^2\geq\frac{1}{2}|A|^2-|B|^2$ and (\ref{weakprestrain}) to show that, for $h$ small enough, 
\begin{equation*}
\int_{\Omega}\dist^2(\nabla_hv_h(x)A_h(x),SO(3))\,\mathrm{d}x\geq \frac{1}{C}\|\nabla v_h\|_{L^2}^2-C,
\end{equation*}
hence the energy bound and (\ref{forcefinite}) imply $\sup_h\|\nabla v_h\|_{L^2(\Omega)}<+\infty$. Since, by the boundary condition \eqref{weakbc}, $v_{h}(0,x_{2},x_{3})$ vanishes with $h$, we deduce that $v_{h}$ is weakly compact in $W^{1,2}(\Omega)$ and that the limit $v$ satisfies 
\begin{equation}\label{vorigin}
v(0,x_{2},x_{3})=0 \text{ a.e. }(x_{2},x_{3})\in S. 
\end{equation}
From the assumption (ii) on $W$ and (\ref{weakprestrain}) we also have that
\begin{equation*}
\int_{\Omega}\dist^2(\nabla_h v_h(x),SO(3))\leq \int_{\Omega}|\nabla_h v_h(x)|^2|I-A_h(x)|^2+\dist^2(\nabla_h v_h(x)A_h(x),SO(3))\,\mathrm{d}x\leq C\,h^2,
\end{equation*}
so that (\ref{distest}) holds and we can apply Proposition \ref{finercompactness}. Up to subsequences the family of rotations $R_h\in C^{\infty}((0,L),SO(3))$ converges weakly in $W^{1,2}((0,L))$ and uniformly on $[0,L]$ to some $R\in W^{1,2}((0,L),SO(3))$. By property $(i)$ in Proposition \ref{finercompactness}, $\nabla_hv_h\to R$ strongly in $L^{2}(\Omega)$. Since $R$ is depending only on $x_{1}$ the limit functions depend only on $x_{1}$. From this fact, the rest of the statement follows. In particular $v(0)=0$ follows from \eqref{vorigin}.
\end{proof}

We will now derive a one-dimensional limit theory. 
Given a matrix $M\in\mathbb{M}^{3\times 3}$ we let $Q_3$ be the quadratic form defined by
\begin{equation*}
Q_3(M):=\frac{\partial^2 W}{\partial F^2}(I)[M,M].
\end{equation*}
Further we set $Q_2:(0,L)\times\mathbb{M}^{3\times 3}_{skew}\to [0,+\infty)$ as
\begin{equation}\label{Q2}
Q_2(x_1,A)=\inf_{\substack{\a\in W^{1,2}(S,\R^3)\\ g\in\R^3}}\int_SQ_3
\left(\left(\begin{array}{c|c|c}
A(x_2e_2+x_3e_3)+g &\partial_{2}\a &\partial_{3}\a
\end{array}\right)+\overline{A}(x)\right)\,\mathrm{d}x_2\mathrm{d}x_3.
\end{equation}
In the next proposition we will collect some useful properties of $Q_{2}$.
\begin{proposition}\label{ondensity}
Let $Q_{2}$ be as in \eqref{Q2}. Then the following holds true:
\begin{itemize}
\item[(i)] $Q_2(x_1,A)$ is given by the equivalent minimum problem
\begin{equation}\label{Q2equiv}
\min_{\substack{\a\in W^{1,2}(S,\R^3)}}\int_SQ_3
\left(\left(\begin{array}{c|c|c}
A(x_2e_2+x_3e_3)+g_{{\rm min}} &\partial_2\a &\partial_3\a
\end{array}\right)+\overline{A}(x)\right)\,\mathrm{d}x_2\mathrm{d}x_3,
\end{equation}
where $g_{{\rm min}}\in L^{\infty}((0,L),\R^3)$ is given by
\begin{equation*}
g_{{\rm min}}(x_{1})=-\left(\int_S\overline{a}_{11}(x_{1},x_{2},x_{3})\,\mathrm{d}x_2\mathrm{d}x_3\right)e_1.
\end{equation*}
\item[(ii)] The problem \eqref{Q2equiv} has a minimizer of the form
$\a_{{\rm min}}=\a^{0}+\hat\a$, where $\a^{0}$ minimizes \eqref{Q2}
for $A=0$ and $\hat\a$ solves
\begin{equation}\label{Q2hom}
\min_{\substack{\a\in W^{1,2}(S,\R^3)}}\int_SQ_3
\left(\left(\begin{array}{c|c|c}
A(x_2e_2+x_3e_3) &\partial_2\a &\partial_3\a
\end{array}\right)\right)\,\mathrm{d}x_2\mathrm{d}x_3.
\end{equation}
\item[(iii)] $Q_2(x_1,A)$ is a quadratic function of the matrix entries of $A$, that is
\begin{equation}\label{decomposition}
	Q_2(x_1,A)=\hat{Q}_2(A)+K(x_1):A+q(x_1)
	\end{equation}
for a coercive quadratic form $\hat{Q}_2$, a matrix-valued function $K\in L^{\infty}((0,L),\mathbb{M}^{3\times 3})$ and a scalar-valued function $q\in L^{\infty}((0,L),\R)$. 
In particular $\hat{Q}_{2}(A)$ only depends on the solution $\hat \a$ of \eqref{Q2hom}.
\end{itemize}
\end{proposition}

\begin{proof}
(i) Given any pair $(g,\a)\in\R^3\times W^{1,2}(S,\R^3)$, for almost every $x_1\in (0,L)$ we may define 
\begin{align*}
\tilde{\a}(x_2,x_3)=\a(x_2,x_3)-x_2p-x_3r
\end{align*}
where $p,r\in\R^{3}$ are defined as
\begin{align*}
&p=p(x_1):=\int_S\partial_2\a\,\mathrm{d}x_2\mathrm{d}x_3+\int_S\left(
\overline{a}_{21}+\overline{a}_{12}, \overline{a}_{22}, \displaystyle\frac{\overline{a}_{23}+\overline{a}_{32}}{2}\right)^T\mathrm{d}x_2\mathrm{d}x_3,
\\
&r=r(x_1):=\int_S\partial_3\a\,\mathrm{d}x_2\mathrm{d}x_3+\int_S\left(
\overline{a}_{31}+\overline{a}_{13},\displaystyle \frac{\overline{a}_{23}+\overline{a}_{32}}{2}, \overline{a}_{33}\right)^T\mathrm{d}x_2\mathrm{d}x_3.
\end{align*}
Then we can write 
\begin{align*}
\left(\begin{array}{c|c|c}
A(x_2e_2+x_3e_3)+g &\partial_2\a &\partial_3\a
\end{array}\right)=&\left(\begin{array}{c|c|c}g-g_{{\rm min}} & p & r\end{array}\right)\\
&+\left(\begin{array}{c|c|c}
A(x_2e_2+x_3e_3)+g_{{\rm min}} &\partial_2\tilde{\a} &\partial_3\tilde{\a}
\end{array}\right)
\\
=&:M_1+M_2.
\end{align*}
Note that $M_1$ depends only on $x_1$ and $M_2$ is constructed in such a way that, by (\ref{sym}), the symmetric part of $M_2+\overline{A}$ has zero average on $S$. Thus, expanding the quadratic form $Q_3(M)$ (which depends only on the symmetric part of $M$) we deduce that
\begin{align*}
&\int_S Q_3\left(\left(\begin{array}{c|c|c}
A(x_2e_2+x_3e_3)+g &\partial_2\a &\partial_3\a
\end{array}\right)+\overline{A}(x)\right)\,\mathrm{d}x_2\mathrm{d}x_3=\int_SQ_3\left(M_1+M_2+\overline{A}(x)\right)\,\mathrm{d}x_2\mathrm{d}x_3\\
&\geq 2\int_S D^2W(I)[{\rm sym}(M_1),{\rm sym}(M_2+\overline{A}(x))]\,\mathrm{d}x_2\mathrm{d}x_3+\int_SQ_3(M_2+\overline{A}(x))\,\mathrm{d}x_2\mathrm{d}x_3\\
& = \int_SQ_3\left(\left(\begin{array}{c|c|c}
A(x_2e_3+x_3e_3)+g_{{\rm min}} &\partial_2\tilde{\a} &\partial_3\tilde{\a}
\end{array}\right)+\overline{A}(x)\right)\,\mathrm{d}x_2\mathrm{d}x_3.
\end{align*}

(ii) As the problem defining $Q_2(x_1,A)$ is convex, we can characterize the minimizers by the Euler-Lagrange equation. Those read as
\begin{align*}
\begin{cases}
\sum_{j=2}^3\partial_j\partial_{m_{ij}}Q_3\left(\left(\begin{array}{c|c|c}A(x_2e_2+x_3e_3)+g_{\rm min} &\partial_2\a &\partial_3\a\end{array}\right)+\overline{A}(x)\right)=0&\mbox{in $S$,}\\
\sum_{j=2}^3\partial_{m_{ij}}Q_3\left(\left(\begin{array}{c|c|c}A(x_2e_2+x_3e_3)+g_{\rm min} &\partial_2\a &\partial_3\a\end{array}\right)+\overline{A}(x)\right)=0 &\mbox{on $\partial S$,}
\end{cases} \quad i=1,\dots,3,
\end{align*}
where we denote by $\partial_{m_{ij}}$ the partial derivatives of the map $M\mapsto Q_3(M)$ with respect to the matrix entries. 
As $\partial_{m_{ij}}Q_3(\cdot)$ is linear for all $i,j$, it follows that $\a_{{\rm min}}=\a^0+\hat{\a}$ solves the Euler-Lagrange equation, where $\a^0$ is the corresponding minimizer for $A=0$ and $\hat{\a}$ is the minimizer for $\overline{A}=g=0$. 
\\
(iii) Note that $\hat{\a}$ depends linearly on $A$ (see also Remark 3.4 in \cite{MM1}) and $\a^0$ is independent of $A$. Introducing the linear map $A\mapsto M(A)(x_2,x_3)=\left(\begin{array}{c|c|c}A(x_2e_2+x_3e_3)&\partial_2\hat{\a}&\partial_3\hat{\a}\end{array}\right)$, the structure of $Q_2(x_1,A)$ then follows by (ii) on setting
\begin{align*}
&\hat{Q}_2(A)=\int_SQ_3\left(M(A)(x_2,x_3)\right)\,\mathrm{d}x_2\mathrm{d}x_3,\\
&K(x_1):A=2\int_S D^2W(I)\left[M(A)(x_2,x_3),\left(\begin{array}{c|c|c}g_{\rm min}&\partial_2\a^0 &\partial_3\a^0\end{array}\right)+\overline{A}(x)\right]\,\mathrm{d}x_2\mathrm{d}x_3,\\
&q(x_1)=\int_SQ_3\left(\left(\begin{array}{c|c|c}g_{\rm min}&\partial_2\a^0 &\partial_3\a^0\end{array}\right)+\overline{A}(x)\right)\,\mathrm{d}x_2\mathrm{d}x_3.
\end{align*} 
The coercivity of $\hat{Q}_2(A)$ can be checked directly, so we omit the details. Measurability of $K$ and $q$ follows from Lemma \ref{measurable} and Fubini's theorem, while the integrability assertions can be deduced from (i), the boundedness of $\overline{A}$ and (\ref{control}) as the latter implies $\|\nabla\a^0\|_{L^2}\leq C$.
\end{proof}

\begin{remark}\label{infprob}
Given a function $A\in L^2((0,L),\mathbb{M}^{3\times 3}_{skew})$, one can choose a function $\a\in L^2(\Omega,\R^3)$ with $\partial_k\a\in L^2(\Omega,\R^3)$ ($k=2,3$) such that $\alpha(x_{1},\cdot)$ solves \eqref{Q2equiv} with $A=A(x_{1},\cdot)$ for $a.e.\, x_{1}\in(0,L)$. This fact is proved in the appendix (see Lemma \ref{measurable}) since in particular the measurability of $\alpha$ in the product space $\Omega=(0,L)\times S$ requires some care.
\end{remark}

Here and henceforth we will use the notation $R(x_1)=(\partial_1 v(x_1)|d_2(x_1)|d_3(x_1))$. 
We define the limit energy $E_0:W^{1,2}(\Omega,\R^3)\times (L^2(\Omega,\R^3))^2\to [0,+\infty]$ as
\begin{equation}\label{deflimit}
E_0(v,d_2,d_3):=
\begin{cases}
\frac{1}{2}\int_0^LQ_2\left(x_1,R^T(x_1)R^{\prime}(x_1)\right)\,\mathrm{d}x_1-\int_{\Gamma}\langle f(x),v(x)\rangle\,\mathrm{d}\mathcal{H}^2 &\mbox{if $(v,d_2,d_3)\in\mathcal{A}$,}
\\
+\infty &\mbox{otherwise.}
\end{cases}
\end{equation}
Note that the definition above is well-posed since, by the separability of $W^{1,2}(S,\R^3)$, the function $Q_2$ is measurable and, since $R\in SO(3)$ almost everywhere, the matrix $R^TR^{\prime}$ is skew-symmetric. 
We further notice that the limit functional does depend only on $R$. In fact, since for admissible functions $v(x_1)=\int_0^{x_1}R(x_1)e_1\,\mathrm{d}x_1$, one can rewrite the force term as
\begin{equation}\label{forcestructure}
\int_{\Gamma}\langle f(x),v(x)\rangle\,\mathrm{d}\mathcal{H}^2=\left\langle\int_{\Gamma}f(L,x_2,x_3)\,\mathrm{d}\mathcal{H}^2,\int_0^LR(x_1)e_1\,\mathrm{d}x_1\right\rangle.
\end{equation}
We now can state the main theorem of this section.

\begin{theorem}\label{1dlimit}
Let $h_j\to 0$ and assume that (\ref{forcescale}) and (\ref{weakprestrain}) hold. Then we have the following $\Gamma$-convergence result:
\begin{itemize}
\item[(i)]For every sequence $v_j\to v$ in $W^{1,2}(\Omega,\R^3)$ and $\frac{1}{h_j}(\partial_2 v_j , \partial_3 v_j)\to (d_2,d_3)$ in $L^2(\Omega)$, it holds that
\begin{equation*}
E_0(v,d_2,d_3)\leq \liminf_j \frac{1}{h_j^2}E_{h_j}(v_j).
\end{equation*}
\item[(ii)] For every $v\in W^{1,2}(\Omega,\R^3)$, $d_2,d_3\in L^2(\Omega,\R^3)$ there exists a sequence $v_j\in W^{1,2}(\Omega,\R^3)$ such that $v_j\to v$ in $W^{1,2}(\Omega)$ and $\frac{1}{h_j}(\partial_2 v_j,\partial_3 v_j)\to (d_2,d_3)$ in $L^2(\Omega)$ fulfilling
\begin{equation*}
\lim_j\frac{1}{h_j^2}E_{h_j}(v_j)= E_0(v,d_2,d_3).
\end{equation*}
\end{itemize}
\end{theorem}

\begin{proof}
The proof follows the lines of Theorem 3.1 in \cite{MM1}. We will therefore only sketch those arguments of the proof which do not require substantial modifications. Since by our assumptions the force term in $E_{h_{j}}$ is continuous with respect to the chosen convergence, we can assume that $f_h=f=0$.\\
{\bf Lower bound (i)}: We can assume that 
\begin{equation*}
\liminf_j\frac{1}{h_j^2}E_{h_j}(v_j)=\lim_j\frac{1}{h_j^2}E_{h_j}(v_j)=C<+\infty.
\end{equation*}
Let $R_{j}=R_{h_{j}}$ with $R_{h_{j}}$ the sequence of smooth rotations constructed in Proposition \ref{compact}. We define the sequence $G_j:\Omega\to \mathbb{M}^{3\times 3}$ as
\begin{equation*}
G_j(x):=\frac{R_j^T(x_1)\nabla_{h_j} v_j(x)-I}{h_j}.
\end{equation*} 
Due to the bound (i) in Proposition \ref{finercompactness}, up to subsequences (not relabeled) we have $G_j\rightharpoonup G$ in $L^2(\Omega)$. Using frame-indifference we rewrite the relevant part of the energy as
\begin{equation}\label{relevantpart}
E_{h_j}(v_j)=\int_{\Omega}W\left(I+h_j\left(R_j^T\nabla_{h_j}v_j\frac{(A_{h_j}-I)}{h_j}+G_j\right)\right)\,\mathrm{d}x.
\end{equation}
Since, by $(i)$ in Proposition \ref{finercompactness}, $R^T_j\nabla_{h_j}v_j\to I$ in $L^2(\Omega)$, we deduce from (\ref{weakprestrain}) that 
\begin{equation*}
R_j^T\nabla_{h_j}v_j\frac{(A_{h_j}-I)}{h_j}+G_j\rightharpoonup \overline{A}+G\quad\text{in }L^2(\Omega).
\end{equation*}
The compactness above together with standard linearization arguments (see the proof of Proposition 6.1 (i) in \cite{FJM}) imply that 
\begin{equation*}
\liminf_j\frac{1}{h_j^2}E_{h_j}(v_j)\geq \frac{1}{2}\int_{\Omega}Q_3(G(x)+\overline{A}(x))\,\mathrm{d}x.
\end{equation*}
It remains to identify the matrix $G$ in terms of $R=(\partial_1 v,d_2,d_3)$. Reasoning as in the proof of Theorem 3.1 (i) in \cite{MM1} one infers that 
\begin{equation}\label{Gcol1}
G(x)e_1=G(x_1,0,0)e_1+R^T(x_1)R^{\prime}(x_1)(x_2e_2+x_3e_3).
\end{equation}
To obtain the remaining columns we argue as follows. Let us define the functions
\begin{equation*}
\alpha_j(x)=\frac{R_j^T(x_1)\frac{1}{h_j}v_j(x)-x_2e_2-x_3e_3}{h_j},
\end{equation*}
so that $\partial_k\alpha_j(x)=G_j(x)e_k$ for $k=2,3$. Defining the partially averaged function $\overline{\alpha}_j(x_1)=\int_S\alpha_j(x)\,\mathrm{d}x_2\mathrm{d}x_3$ we infer from the Poincar\'e inequality that, for almost every $x_1\in (0,L)$,
\begin{equation*}
\int_S|\alpha_j(x)-\overline{\alpha}_j(x_1)|^2\,\mathrm{d}x_2\mathrm{d}x_3\leq C\int_S|\partial_2\alpha_j|^2+|\partial_3\alpha_j|^2\,\mathrm{d}x_2\mathrm{d}x_3.
\end{equation*}
Integrating the above inequality over $(0,L)$ yields
\begin{equation*}
\int_{\Omega}|\alpha_j(x)-\overline{\a}_j(x)|^2\,\mathrm{d}x\leq C\int_{\Omega}|\partial_2\alpha_j|^2+|\partial_3\alpha_j|^2\,\mathrm{d}x\leq C,
\end{equation*} 
where we used the $L^2$-bound on $G_j$. Passing to a subsequence we may suppose that $\alpha_j-\overline{\a}_j\rightharpoonup\a$ in $L^2(\Omega)$ for some $\alpha\in L^2(\Omega,\R^3)$. Note that in the sense of distributions we have $\partial_k\alpha=Ge_k$. Moreover, for almost every $x_1\in (0,L)$ we have that the map $\alpha_{x_{1}}:(x_2,x_3)\to\alpha(x_1,x_2,x_3)$ belongs to $W^{1,2}(S,\R^3)$. Therefore, for almost every $x_{1}\in (0,L)$ the pair $(G(x_{1},0,0)e_{1},\alpha_{x_{1}})\in\R^{3}\times W^{1,2}(S,\R^3)$ is admissible in the minimum problem \eqref{Q2} for $A=R^{T}(x_{1})R'(x_{1})$. By this, integrating with respect to $x_{1}$ and using \eqref{Gcol1}, we get
\begin{equation*}
\frac{1}{2}\int_{\Omega}Q_3(G(x)+\overline{A}(x))\,\mathrm{d}x\geq\frac{1}{2}\int_0^LQ_2\left(x_1,R^T(x_1)R^{\prime}(x_1)\right)\,\mathrm{d}x_1,
\end{equation*}
thus concluding the proof of the lower bound.\\

{\bf Upper bound (ii)}: For any $(v,d_2,d_3)\in\mathcal{A}$, by Proposition \ref{ondensity} (i) and Remark \ref{infprob} we can choose $\a\in L^2(\Omega,\R^3)$ with $\partial_k\a\in L^2(\Omega,\R^3)$ for $k=2,3$ such that (abusing notation)
\begin{equation*}
E_0(R)=\frac{1}{2}\int_{\Omega}Q_3
\left(\left(\begin{array}{c|c|c}
R^T(x_1)R^{\prime}(x_1)(x_2e_2+x_3e_3)+g_{\rm{min}}(x_1) &\partial_2\a &\partial_3\a
\end{array}\right)+\overline{A}(x)\right)\,\mathrm{d}x.
\end{equation*}
We will approximate $R,\a$ and $g$ in a similar way as done in the proof of Theorem 3.1 (ii) in \cite{MM1}. First we define $\gamma_n(x_1)=\int_0^{x_1}g_n(s)\,\mathrm{d}s$, where $g_n\in C([0,L],\R^3)$ is a sequence of functions converging to $g_{\rm min}$ in $L^2((0,L))$. In order to approximate $\a$ we consider the standard approximation $\a_n$ obtained by convolution, where we make use of Remark \ref{ext} to have convergence up to the boundary. Then it holds that $\a_n\to\a$ and $\partial_k\a_n\to \partial_k\a$ in $L^2(\Omega)$ for $k=2,3$. We can furthermore keep the boundary condition (\ref{weakbc}) replacing $\a_{n}$ by $\tilde{\a}_n(x)=\a_n(x)\theta_n(x_1)$, where $\theta_n\in C^{\infty}([0,1]),[0,1])$ is a suitable cut-off function.
This still preserves the convergence of $\tilde \alpha_{n}$ and of its partial derivatives $\partial_k\tilde{\a}_n$ in $L^2(\Omega)$ to $\alpha$ and $\partial_{k}\alpha$ for $k=2,3$, respectively. Note that since $W^{1,2}((0,L))$ is compactly embedded into $C([0,L])$ we can construct by projection onto $SO(3)$ also a sequence of smooth functions $R_n\in C^{\infty}([0,L],SO(3))$ such that $R_n\to R$ in $W^{1,2}(0,L)$ and uniformly on $[0,L]$. Setting
\begin{equation*}
y_n(x_1):=\int_0^{x_1}R_n(s)e_1\,\mathrm{d}s,\quad\quad (d_k)_n(x_1)=R_n(x_1)e_k,\quad k=2,3,
\end{equation*}
then for fixed $n$ we have $y_n,(d_k)_n\in C^{\infty}([0,L],\R^3)$ for $k=2,3$. Moreover, up to a subsequence (not relabeled) we may suppose by continuity that
\begin{equation*}
\frac{1}{2}\int_{\Omega}Q_3
\left(\left(\begin{array}{c|c|c}
R_{n}^T(x_1)R^{\prime}_{n}(x_1)(x_2e_2+x_3e_3)+g_{n}(x_1) &\partial_2\tilde{\a}_{n} &\partial_3\tilde{\a}_{n}
\end{array}\right)+\overline{A}(x)\right)\,\mathrm{d}x\leq E_0(R)+\frac{1}{n}.
\end{equation*} 
We set
\begin{equation}\label{recoveryplan}
v^n_h(x):=y_n(x_1)+hx_2(d_2)_n(x_1)+hx_3(d_3)_n(x_1)+h\gamma_n(x_1)+h^2R_n(x_1)\tilde{\a}_n(x).
\end{equation}
The regularity of $v^n_{h}$ allows a Taylor expansion of the energy density around the identity. By (\ref{weakprestrain}) and the dominated convergence theorem one obtains (see \cite{MM1} for more details)
\begin{equation*}
\lim_{h\to 0}\frac{1}{h^2}E_{h}(v^n_h)=\frac{1}{2}\int_{\Omega}Q_3\left(\left(\begin{array}{c|c|c} R_n^T(x_1)R_n^{\prime}(x_1)(x_2e_2+x_3e_3)+g_n(x_1) &\partial_2\tilde{\a}_n&\partial_3\tilde{\a}_n\end{array}\right)+\overline{A}(x)\right)\,\mathrm{d}x.
\end{equation*}
Given any sequence $h_j\to 0$, there exists a subsequence $h_{j_n}$ such that, if we define the sequence $v^n_{h}$ as in (\ref{recoveryplan}) with $h\leq h_{j_n}$, it fulfills (\ref{weakbc}) and by the previous reasoning it holds that $\frac{1}{h^2}E_{h}(v^n_h)\leq E_0(R)+\frac{2}{n}$. Finally we set $v_j:=v_{h_j}^{n}$ if $j_n\leq j<j_{n+1}$. Then $\limsup_j\frac{1}{h_j^2}E_{h_j}(v_j)\leq E_0(v,d_2,d_3)$ and $v_j\to v$ in $W^{1,2}(\Omega)$ as well as $1/{h_j}(\partial_2 v_j,\partial_3 v_j)\to (d_2,d_3)$ in $L^2(\Omega)$.
\end{proof}
\begin{remark}\label{bulkforces}
The $\Gamma$-convergence result still holds if we consider bulk forces of the type $\int_{\Omega}\langle f_h(x),v(x)\rangle\,\mathrm{d}x$ instead of forces acting on the end. The proof is analogous.
\end{remark}
	

\subsection{The isotropic case: intrinsically curved rods}
In this section we treat the case of isotropic materials showing that the asymptotic formula \eqref{Q2} can be further simplified to a Kirchhoff rod-model with intrinsic curvature if we assume for the prestrain matrix $\overline A_{h}$ a special structure arising from a two-layer model (see \eqref{twolayer}).  We begin with some general consideration about isotropic materials. Through all this section we assume that the isotropy assumption
\begin{enumerate}
	\item[(iv)] $W(FR)=W(F)\quad\forall R\in SO(3)$
\end{enumerate}
is satisfied. It is well-known that in this case the quadratic form $Q_{3}$ in \eqref{Q2} simplifies to 
\begin{equation}\label{simplestructure}
Q_3(M)=2\mu|{\rm sym}(M)|^2+\lambda \text{tr}(M)^2,
\end{equation}
for two positive constants $\mu,\lambda$ that are known as Lam\'e constants.

For the sake of simplicity, we assume that the prestrain describes locally an incompressible deformation, that means $\text{det}(\overline A_h)=1$. This implies that
\begin{equation}\label{tracezero}
\text{tr}(\overline{A}(x))=0\quad\text{a.e. in }\Omega.
\end{equation} 
We recall that, due to Proposition \ref{ondensity} (ii), the problem \eqref{Q2equiv} has a minimizer of the form
$\a_{{\rm min}}=\a^{0}+\hat\a$, where $\a^{0}$ minimizes \eqref{Q2}
for $A=0$ and $\hat\a$ is a solution of \eqref{Q2hom}. Since the coercive quadratic form $\hat{Q}_{2}$ in the decomposition \eqref{decomposition} only depends on $\hat \a$, the direct computation in \cite[Remark 3.5]{MM1} shows that
\begin{equation}\label{quadraticpart}
\hat{Q}_2(A)=\mu\frac{3\lambda+2\mu}{\lambda+\mu}\left(a_{12}^2\int_Sx_2^2+a_{13}^2\int_Sx_3^2\right)+\mu \tau_{\scriptscriptstyle{S}} a_{23}^2,
\end{equation}
where the constant $\tau_{\scriptscriptstyle{S}}$ is the so-called torsional rigidity. Introducing the torsion function $\varphi$ of $S$, defined as solution of
\begin{equation}\label{torsionpde}
\begin{cases}
\Delta\varphi=0 &\mbox{in $S$,}\\
\partial_{\nu}\varphi=-\langle(x_3,-x_2)^{T},\nu\rangle&\mbox{on $\partial S$,}
\end{cases}
\end{equation}
the torsional rigidity is given by
\begin{equation}\label{torsionconstant}
\tau_{\scriptscriptstyle{S}}:=\int_S(x_2^2+x_3^2+x_3\partial_2\varphi -x_2\partial_3\varphi )=\int_S [(x_3+\partial_2\varphi)^2+(\partial_3\varphi-x_2)^2]\,,
\end{equation}
where the equivalence between the two definitions follows from \eqref{torsionpde} and integration by parts.

On the other hand, using the trace-free condition \eqref{tracezero}, a direct computation shows that $\alpha^{0}$ solves the following 
system of Euler-Lagrange equations:
\begin{eqnarray}\label{Euler}
\begin{cases}
-\Delta \a_1=\text{div}(\overline{a}_{12}+\overline{a}_{21},\overline{a}_{13}+\overline{a}_{31}) &\mbox{in $S$,}\\ 
\text{div}((2\mu+\lambda)\partial_2\a_2+\lambda\partial_3\a_3,\mu\partial_2\a_3+\mu\partial_3\a_2)=\text{div}(-2\mu\overline{a}_{22}-\lambda g_1,-\overline{a}_{23}-\overline{a}_{32})&\mbox{in $S$,}\\
\text{div}(\mu\partial_2\a_3+\mu\partial_3\a_2,\lambda\partial_2\a_2+(2\mu+\lambda)\partial_3\a_3)=\text{div}(-\overline{a}_{23}-\overline{a}_{32},-2\mu \overline{a}_{33}-\lambda g_1)&\mbox{in $S$,}\\
\partial_{\nu}\a_1=-\langle(\overline{{a}}_{12}+\overline{a}_{21},\overline{a}_{13}+\overline{a}_{31})^{T},\nu\rangle &\mbox{on $\partial S$,}\\
\langle((2\mu+\lambda)\partial_2\a_2+\lambda\partial_3\a_3,\mu\partial_2\a_3+\mu\partial_3\a_2)^{T},\nu\rangle=\langle(-2\mu \overline{a}_{22}-\lambda g_1,-\overline{a}_{23}-\overline{a}_{32})^{T},\nu\rangle&\mbox{on $\partial S$,}\\
\langle(\mu\partial_2\a_3+\mu\partial_3\a_2,\lambda\partial_2\a_2+(2\mu+\lambda)\partial_3\a_3)^{T},\nu\rangle=\langle(-\overline{a}_{23}-\overline{a}_{32},-2\mu \overline{a}_{33}-\lambda g_1)^{T},\nu\rangle&\mbox{on $\partial S$.}
\end{cases}
\end{eqnarray}

We now focus on the particular case when the prestrain has a special two-layer structure. This assumption is motivated by the experiments in \cite{Plos}. There the authors stretch a shorter thin piece of rubber with rectangular cross section until it has the same length as a second one and then they glue them along one side. To describe the effects when an additional force is applied at one end, they use a Kirchhoff model with intrinsic curvature. Here we derive this model from three-dimensional elasticity. The result we present does only depend on the structure of the limit function $\overline{A}$, so that more general hypotheses on the prestrain are possible. For the sake of simplicity we suppose the prestrain $\overline A_{h}$ to be of the form
\begin{equation}\label{twolayer}
\overline A_h(x):=\begin{cases}
{\rm{diag}}(1+h\chi,\frac{1}{\sqrt{1+h\chi}},\frac{1}{\sqrt{1+h\chi}}) &\mbox{if $x_3>0$,}\\
I &\mbox{if $x_3<0$,}
\end{cases}
\end{equation}
where $\chi>0$ describes the effective strength of the stretching. If $\overline A_h$ is of the form (\ref{twolayer}), then $A_h$ fulfills (\ref{weakprestrain}) with
\begin{equation}\label{prestrainexample}
\overline{A}(x):=\begin{cases}
{\rm{diag}}(\chi,-\frac{\chi}{2},-\frac{\chi}{2}) &\mbox{if $x_3>0$,}\\
0 &\mbox{if $x_3<0$}
\end{cases}
\end{equation}
and in particular (\ref{tracezero}) holds. We set 
\begin{equation*}
S^{+}:=S\cap\{x_3>0\}
\end{equation*}
and, in addition to (\ref{sym}), we further assume the following symmetry condition:
\begin{equation}\label{upperlayer}
\int_{S^{+}}x_2\,\mathrm{d}\mathcal{H}^2=0.
\end{equation} 

\begin{proposition}\label{curvaturelimit}
Assume that (\ref{prestrainexample}) and (\ref{upperlayer}) hold and that $W$ is isotropic. Then, there exist positive coefficients $c_{12}, c_{13}, c_{23}$ such that, up to an additive constant, the density $Q_2$ in \eqref{Q2} is given by
\begin{equation}\label{abstractdensity}
Q_2(A)=c_{12}a_{12}^2+c_{13}(a_{13}-k)^2+c_{23}a_{23}^2,
\end{equation}
where
\begin{equation}\label{curvature-prestrain}
k=\chi \frac{\int_{S^{+}}x_3\,\mathrm{d}\mathcal{H}^2}{\int_Sx_3^2\,\mathrm{d}\mathcal{H}^2}
\end{equation}
is the intrinsic curvature.
\end{proposition}

\begin{proof}
From \eqref{Q2} and (\ref{prestrainexample}) we deduce that limit density does not depend on $x_1$. Referring to \eqref{decomposition}, we notice that
$\hat Q_{2}$ is given by \eqref{quadraticpart}, $q$ affects the energy only by an additive constant, while the linear term is given by
$K:A$ with $K$ independent of $x_{1}$ because of the previous considerations. Using the notation from the proof of Proposition \ref{ondensity}, (\ref{simplestructure}) and (\ref{tracezero}) yield
\begin{align}\label{linearpart}
K:A=&4\mu\int_S{\rm sym}\left(M(A)(x_2,x_3)\right):{\rm sym}\left(\left(\begin{array}{c|c|c}g_{\rm min}&\partial_2\a^0&\partial_3\a^0\end{array}\right)+\overline{A}(x)\right)\,\mathrm{d}\mathcal{H}^2\nonumber\\
&+2\lambda\int_S{\rm tr}(M(A)(x_2,x_3))\,{\rm tr}\left(\left(\begin{array}{c|c|c}g_{\rm min}&\partial_2\a^0&\partial_3\a^0\end{array}\right)\right)\,\mathrm{d}\mathcal{H}^2.
\end{align}
First notice that by \cite[Remark 3.5]{MM1} the second and third components of $\hat{\a}$ are given by
\begin{equation}\label{explicitsol}
\begin{pmatrix}
\hat{\a}_2(x_2,x_3)\\ \hat{\a}_3(x_2,x_3)\end{pmatrix}=-\frac{1}{4}\frac{\lambda}{\lambda+\mu}\begin{pmatrix}a_{12}x^2_2-a_{12}x_3^2+2a_{13}x_2x_3\\
-a_{13}x_2^2+a_{13}x_3^2+2a_{12}x_2x_3
\end{pmatrix}.
\end{equation}
It follows that ${\rm sym}(M(A)(\cdot))_{23}={\rm sym}(M(A)(\cdot))_{32}=0$. Moreover, the solution $\a^{0}$ to the system \eqref{Euler} can be taken with first component $\a^{0}_1=0$ since the non-diagonal terms in $\overline{A}$ are zero, and so are then the right-hand sides in the equation and the boundary condition for $\a^{0}_1$. By Proposition \ref{ondensity} and \eqref{prestrainexample} we further have $g_{{\rm min}}=-\chi {\mathcal L}^{2}(S^{+})e_1$. From all this facts we conclude that in the matrix scalar-product in (\ref{linearpart}) only the diagonal entries give a contribution.  Thus, plugging in (\ref{explicitsol}), the expression in (\ref{linearpart}) simplifies as
\begin{align*}
K:A=&4\mu\int_S (a_{12}x_2+a_{13}x_3)(-\chi {\mathcal L}^{2}(S^{+})+\overline{a}_{11}(x))\,\mathrm{d}\mathcal{H}^2\\
&-\int_S\frac{2\mu\lambda}{\lambda+\mu}(a_{12}x_2+a_{13}x_3)(\partial_2\a^0_2+\partial_3\a^0_3+\overline{a}_{22}(x)+\overline{a}_{33}(x))\,\mathrm{d}\mathcal{H}^2 \\
&+\int_S\frac{2\lambda\mu}{\lambda+\mu}(a_{12}x_2+a_{13}x_3)(\partial_2\a^0_2+\partial_3\a^0_3-\chi {\mathcal L}^{2}(S^{+}))\,\mathrm{d}\mathcal{H}^2
\\
=&2\chi\mu \frac{3\lambda+2\mu}{\lambda+\mu}\int_{S^+}(a_{12}x_2+a_{13}x_3)\,\mathrm{d}\mathcal{H}^2=2\chi\mu \frac{3\lambda+2\mu}{\lambda+\mu}\int_{S^{+}}a_{13}x_3,
\end{align*}
where we used (\ref{sym}), \eqref{prestrainexample} and (\ref{upperlayer}). Note that the explicit expression of $\a^0_2$ and $\a^0_3$ is not needed in the previous calculation, since the corresponding term cancels. The general formula (\ref{abstractdensity}), as well as the formula for $k$, now follows from the above computation and (\ref{quadraticpart}) by completing the squares.
\end{proof}

\begin{remark}
The coefficients of the density (\ref{abstractdensity}) can also be easily obtained from (\ref{quadraticpart}). We will use their explicit form in the final section.
\end{remark}

\begin{remark}\label{curvmin}
The function $Q_2$ in (\ref{abstractdensity}) is minimized pointwise by skew-symmetric matrices of the form
\begin{equation*}
A_{{\rm min}}=k(e_1\otimes e_3-e_3\otimes e_1).
\end{equation*}
Defining the function $(v,d_2,d_3)\in\mathcal{A}$ implicitly via
\begin{equation}\label{curvemin}
R(x_1):=\begin{pmatrix} \cos(kx_1) & 0 &\sin(kx_1)\\ 0 & 1& 0\\-\sin(kx_1) & 0 & \cos(kx_1)\end{pmatrix},
\end{equation}
we have 
\begin{equation*}
R^T(x_1)R^{\prime}(x_1)=A_{{\rm min}}.
\end{equation*}
Hence, with the boundary condition (\ref{weakbc}) or (\ref{clamp}) we have a curved beam as global minimizer. On the other hand, if $(v,d_2,d_3)\in\mathcal{A}$ is a minimizer, then $R^{\prime}=RA_{{\rm min}}$ has a unique solution up to a constant rotation prescribing the initial value $R(0)$. We infer that up to rotation, the curved beam constructed above is the unique minimizer when there are no external forces. By the general theory of $\Gamma$-convergence, it follows that for $h$ small enough, minimizers of the energy $E_h$ are close to a curved beam in the chosen topology. 
\end{remark}

\subsection{Different kind of boundary conditions}\label{sect-boundary}
In this paragraph we shortly discuss the effect of the clamped and weak-clamped boundary conditions \eqref{clamp}-\eqref{rigid} on the limit model. While the limit formula for the energy remains the same as in \eqref{deflimit}, its domain takes instead into account some Dirichlet-type boundary conditions for the strain as the next lemma shows.

\begin{lemma}\label{boundaryeffects}
Let $v_h,v$ and $R=(\partial_1v,d_2,d_3)$ be as in Proposition \ref{compact}.
\begin{itemize}
\item[(i)] If $v_h$ fulfills $(\ref{clamp})$, then $R(0)=R_{0}$,
\item[(ii)] If $v_h$ fulfills $(\ref{clamp})$ and \eqref{clamp-clamp}, then $R(0)=R_{0}$ and $R(L)=R_{L}$,
\item[(iii)] If $v_{h}$ fulfills $(\ref{rigid})$, then $R(0)e_{1}=R(L)e_{1}=e_{1}$.
\end{itemize}
\end{lemma}

\begin{proof}
Claim (i) follows immediately from Proposition \ref{finercompactness} and Remark \ref{rem-clamped}.

We now prove (iii) as the argument for (ii) turns out to be a special case. We start with the boundary condition at $x_1=0$. Let $M_{0,h}$ be the corresponding rotations in (\ref{rigid}). Defining $m^0_h=M^T_{0,h}v_h$, it follows from frame indifference that $E_h(m^0_h)\leq Ch^2$ and $m^0_h$ fulfills 
\begin{equation*}
m^0_h(0,x_2,x_3)=(0,hx_2,hx_3)^T.
\end{equation*}
Applying Proposition \ref{compact} and Proposition \ref{finercompactness} we infer that (up to subsequences) $\nabla_hm^0_h$ converges to a function $\tilde R\in W^{1,2}((0,L),SO(3))$ with $\tilde R(0)=I$. On the other hand, by compactness of $v_h$, we have that the sequence of matrices $M_{0,h}$ converges to $M_0\in SO(3)$ with $M_0e_1=e_1$,  due to \eqref{rigid}.
By construction it holds: $R(x_1)=M_0 \tilde R(x_1)$, so that  we deduce $R(0)e_1=M_0\tilde R(0)e_1=e_1$.

To prove that $R(L)e_{1}=e_{1}$, we set $F=\text{diag}(-1,-1,1)\in SO(3)$ and $\Omega^{\prime}=F\Omega+Le_1$. Note that the cross section of $\Omega^{\prime}$ still has unitary area and fulfills (\ref{sym}). We define a new sequence $m^L_h:\Omega^{\prime}\to\R^3$ 
\begin{equation*}
m^L_h(x_1,x_2,x_3)= FM_{L,h}^T\left(v_h(Fx+Le_1)-\int_Sv(L,x_2,x_3)\,\mathrm{d}\mathcal{H}^2\right).
\end{equation*}
Then we have $m_h^L(0,x_2,x_3)=(0,hx_2,hx_3)^T$ and arguing as before we deduce that (up to subsequences) $\nabla_hm^L_h$ converges to a function  $\hat R\in W^{1,2}((0,L),SO(3))$ with $\hat R(0)=I$. By construction it holds now $R(x_1)=F^T M_L \hat R (L-x_1)F$, with $M_L$ the limit of the rotations $M_{L,h}$. It therefore holds $M_L e_1=e_1$, which implies the conclusion by a direct computation.
\end{proof}
As a corollary of Theorem \ref{1dlimit} we have the following convergence result about prescribed boundary conditions.
\begin{corollary}\label{convbc}
Let $h_j\to 0$ and assume that (\ref{weakprestrain}) and (\ref{forcescale}) hold. It holds that	
\begin{itemize}
	\item[(i)] under the boundary conditions (\ref{clamp}), the $\Gamma$-limit of $E_{h}/h^{2}$ is finite only for $(v,d_2,d_3)\in\mathcal{A}$ with $R(0)=R_{0}$,
	\item[(ii)] under the boundary conditions (\ref{clamp}) and (\ref{clamp-clamp}), the $\Gamma$-limit of $E_{h}/h^{2}$ is finite only for $(v,d_2,d_3)\in\mathcal{A}$ with $R(0)=R_{0}$ and $R(L)=R_{L}$.
	\item[(iii)] under the boundary conditions (\ref{rigid}), the $\Gamma$-limit of $E_{h}/h^{2}$ is finite only for $(v,d_2,d_3)\in\mathcal{A}$ with $R(0)e_{1}=R(L)e_{1}=e_{1}$. 
\end{itemize}
In all cases the limit energy is given by $E_0(v,d_2,d_3)$ on the corresponding domain. 	
\end{corollary}

\begin{proof}
{\bf Lower bound}: In all the considered cases the lower bound inequality is a consequence of Theorem \ref{1dlimit} and of Lemma \ref{boundaryeffects}.

{\bf Upper bound} In all cases it suffices to modify the recovery sequence given by Theorem \ref{1dlimit} close to $\{0,L\}\times S$. We demonstrate it for (ii) and (iii). As the construction will be local the argument for (i) is contained in (ii).
 
(ii): Let $(v,d_2,d_3)\in\mathcal{A}$ such that $R(0)=R_0$ and $R(L)=R_L$.  By well-known properties of Sobolev functions with prescribed trace, there exists a sequence $R_n  \in C^\infty((0,L), SO(3))$ converging to $R$ in $W^{1,2}$ which additionally satisfies $R_n=R_0$ in $(0,\delta_n)$ as well as $R_n=R_L$ in $(L-\delta_n,L)$ for some $\delta_n>0$. We now construct the recovery sequence $v_h$ as in the proof of Theorem \ref{1dlimit}, that is
\begin{equation}\label{recoveryj}
v^n_h(x)=y_n(x_1)+hx_2(d_{2})_n(x_1)+hx_3(d_3)_n(x_1)+h\gamma_n(x_1)+h^2R_n(x_1)\tilde{\alpha}_n(x)\,.
\end{equation} 
By (\ref{sym}), this sequence fulfills the boundary condition (\ref{clamp}) and (\ref{clamp-clamp}) if 
\begin{equation*}
\begin{split}&(y_n,(d_2)_n,(d_3)_n)(0)=(0,R_0e_2,R_0e_3),\quad\tilde{\alpha}_n(0,x_2,x_3)=0,\quad\gamma_j(0)=0,\\
&((d_2)_n,(d_3)_n)(L)=(R_Le_2,R_Le_3),\quad\tilde{\alpha}_n(L,x_2,x_3)=0.
\end{split}
\end{equation*} 
One can directly check that all the above equalities, up to the last one, hold now by construction. For the condition $\tilde{\alpha}_n(L,x_2,x_3)=0$ it suffices to modify the sequence $\tilde \alpha_n$ in a neighborhood of $L$ by multiplication with a cutoff only depending on $x_1$. As in the proof of Theorem \ref{1dlimit} this does not affect neither the convergence of $\tilde \alpha_n$ nor of the derivatives in the $x_2$ and $x_3$ directions.

(iii): In order to fulfill the boundary condition (\ref{rigid}) for a recovery sequence of the type (\ref{recoveryj}), we can ensure that $\{(d_2)_n(0),(d_3)_n(0),(d_2)_n(L),(d_3)_n(L)\}\in e_1^{\perp}$ again by choosing a suitable approximation $R_n$ of a Sobolev function with prescribed trace. 

In both cases the modification of the recovery sequence does not change the argument for the energy estimates, so that we conclude.

\end{proof}

\section{Approximation of isolated local minimizers}\label{convisolated}
Having identified the $\Gamma$-limit of the three-dimensional model in the chosen scaling regime, we can apply the well-known results on convergence of (almost) global minimizers. However, in practice also local minimizers are of interest. In this subsection we prove that under some lower semicontinuity assumptions for the sequence of energies $E_{h}$ any strict local minimizer $R$  of the limit energy $E_0$ with respect to the $W^{1,2}$-topology can be obtained as limit of a sequence of local minimizers $v_h$ of the three-dimensional energy $E_h$. Local minimality of the deformations $v_h$ has in this case to be understood with respect to the $L^2$-topology on the scaled gradients $\nabla_h v$, which, as seen in Proposition \ref{compact}, arises naturally on the sublevel sets of $\frac1{h^2} E_h$.

To prove this result, we will mainly rely on the general theory on local minimizers and $\Gamma$-convergence (see \cite[Chapter 5]{BrLoc}). At least at a first glance, this would however need a weaker topology and therefore a stronger notion of minimality, quite hard to be checked in practice. The key observation to overcome this difficulty is that the $L^2$-convergence together with the quadratic structure of the energy $E_0$ imply $W^{1,2}$-convergence in the limit variable $R$ (see Proposition \ref{weak-strong} below). The advantage of this approach is that we can provide simple criteria based on the second variation of the limit functional $E_0$ to check local minimality of $R$ exactly in the $W^{1,2}$-topology, thus completely bridging the gap between the $3$-dimensional and the one-dimensional energy.

We start by fixing some  notation in order to deal also with the additional constraints in Corollary \ref{convbc}. We will see the limit energy $E_0$ as defined on $W^{1,2}((0,L), SO(3))$, since it depends only on $R$. For $\mathcal A_{SO(3)}$ as in \eqref{admissible} we will denote with $\mathcal{A}_h\subseteq \mathcal{A}_{SO(3)}$ the domain of the energy $E_h$. This notation accounts both for the case where only \eqref{weakbc} is assumed and then $\mathcal{A}_h=A_{SO(3)}$, as well as for additionally considering the boundary conditions \eqref{clamp}, \eqref{clamp}-\eqref{clamp-clamp}, and \eqref{rigid}, respectively. Accordingly, the domain $\hat{\mathcal A}$ of the limit energy $E_0$ will be either simply given by $W^{1,2}((0,L), SO(3))$, or will also incorporate the Dirichlet-type boundary conditions in (i), (ii), or (iii) in Corollary \ref{convbc}, respectively.

The next proposition is of high importance since it guarantees that strict local minimizers of $E_0$ in the strong topology are also strict local minimizers in a weaker topology.

\begin{proposition}\label{weak-strong}
	Let $E_0$ be defined as in \eqref{deflimit} and let $R$ be a strict local minimizer of $E_0$ in $\hat{\mathcal A}$ with respect to the strong topology of $W^{1,2}((0,L), SO(3))$. Then $R$ is also a strict local minimizer of $E_0$ in $\hat{\mathcal A}$ with respect to the $L^2$-topology.
\end{proposition} 

\begin{proof}
	We argue by contradiction. Assume that there is a sequence $R_n\in W^{1,2}((0,L),SO(3))$ such that $R_n\to R$ in $L^2((0,L))$ and $E_0(R_n)\leq E_0(R)$. It is enough to show that $R_n\to R$ strongly in $W^{1,2}((0,L))$. 
	By taking advantage of Proposition \ref{ondensity} (iii) we have
	\begin{equation*}
	Q_2(x_1,A)=\hat{Q}_2(A)+K(x_1):A+q(x_1),
	\end{equation*}
	where $\hat{Q}_{2}$ is a coercive quadratic form and $K,q\in L^{\infty}((0,L))$. By this, the definition of $E_0$ and the equality $|R^TR^{\prime}|=|R^{\prime}|$, we obtain that $R_n\rightharpoonup R$ in $W^{1,2}((0,L))$. From the assumption $E_0(R_n)\leq E_0(R)$ and lower semicontinuity we get to
	\begin{equation*}
	\lim_nE_0(R_n)=E_0(R)\,.
	\end{equation*}
	Since the force term converges we deduce that also 
	\begin{equation}\label{energylimit}
	\lim_n\int_0^LQ_2(x_1,R^T_n(x_1)R^{\prime}_n(x_1))\,\mathrm{d}x_1=\int_0^LQ_2(x_1,R^T(x_1)R^{\prime}(x_1))\,\mathrm{d}x_1.
	\end{equation}
	 Applying $L^2$-weak convergence of $R_n^TR_n^{\prime}$ we further infer from (\ref{energylimit}) that
	\begin{equation*}
	\lim_n\int_0^L\hat{Q}_2(R^T_n(x_1)R^{\prime}_n(x_1))\,\mathrm{d}x_1=\int_0^L\hat{Q}_2(R^T(x_1)R^{\prime}(x_1))\,\mathrm{d}x_1.
	\end{equation*}
	By the properties of $\hat{Q}_{2}$ the convergence above implies that $R_n^TR^{\prime}_n\to R^TR^{\prime}$ also strongly in $L^2((0,L))$. Again using the equalities $|R_n^TR_n^{\prime}|=|R_n^{\prime}|$ and $|R^TR^{\prime}|=|R^{\prime}|$ and uniform convexity, it follows that $R^{\prime}_n\to R^{\prime}$ strongly in $L^2((0,L))$. This proves the claim.
\end{proof}

From the previous proposition we deduce the convergence to isolated local minimizers.
\begin{theorem}\label{conv-loc}
Assume that the functional $E_{h}$ defined in (\ref{defenergy}) is lower semicontinuous with respect to weak convergence in $W^{1,2}(\Omega,\R^3)$. Moreover let $E_0$ be defined as in \eqref{deflimit} and $R=(\partial_1 v|d_2|d_3)\in W^{1,2}((0,L),SO(3))$ be a strict local minimizer of $E_0$ in $\hat{\mathcal A}$ with respect to the strong $W^{1,2}$-topology. Then there exists a sequence $v_h$ of local minimizers of $E_h$ in $\mathcal A_h$ such that $v_h\to v$ strongly in $W^{1,2}(\Omega,\R^3)$ and $\nabla_h v_h\to R$ strongly in $L^2(\Omega,\mathbb{M}^{3\times 3})$.
\end{theorem}

\begin{proof}
The proof follows in the footsteps of \cite[Theorem 5.1]{BrLoc}.
By Proposition \ref{weak-strong} we deduce from the assumptions that $R$ is a strict local minimizer in the $L^2$-topology. Hence there exists $\delta>0$ such that $E_0(R')>E_0(R)$ for all $\|R'-R\|_{L^2}\leq\delta$.  We now set 
$$
U^h_{\delta}:=\{v\in \mathcal A_h\;:\|\nabla_hv-R\|_{L^{2}}<\delta\}\,.
$$
Note that, by compactness of $SO(3)$, for fixed $h$ the set $\mathcal A_h$ is closed with respect to weak convergence in $W^{1,2}(\Omega,\R^3)$. It also follows that the weak and strong closure of $U^h_{\delta}$ coincide by convexity. Since the functionals $E_{h}$ are lower semicontinuous and coercive with respect to weak convergence by our assumptions, there exists a sequence $v_h$ such that $v_h$ is a minimizer of $E_{h}$ on $\overline{U^h_{\delta}}$. By Proposition \ref{compact}, or Lemma \ref{boundaryeffects}, the scaled gradients $\nabla_h v_h$ converge to some $R'\in\hat{\mathcal{A}}$ strongly in $L^2$. It follows that $\|R'-R\|_{L^2}\leq\delta$. 

Assume now by contradiction that $R'\neq R$. We infer that	
\begin{equation*}
E_0(R)<E_0(R')\leq\liminf_{h}\frac{1}{h^2}E_h(v_h)\leq\limsup_{h}\frac{1}{h^2}\inf_{v\in U^h_{\delta}}E_h(v)\le\limsup_{h}\frac{1}{h^2} E_h(\hat v_h)\leq E_0(R),
\end{equation*}
where $\hat v_h$ is a recovery sequence for $R$ in the sense of Theorem \ref{1dlimit}. Thus we obtain a contradiction and $R'=R$. It follows that $\nabla_h v_h\in U^h_{\delta}$ for $h$ small enough, which is an open set of $\mathcal{A}_h$ in the $W^{1,2}(\Omega,\R^3)$-topology. Hence $v_h$ is indeed a local minimizer. The convergence of $v_h \to v$ follows from the convergence of $\nabla_h v_h$ to $R$ and the condition \eqref{weakbc} which rules out translations.
\end{proof}


In the rest of the Section we aim to show that a solution $R$ of the Euler-Lagrange equations for $E_0$ with positive second variation is a strict local minimizer of $E_0$ with respect to the strong $W^{1,2}$-topology. This result cannot be directly deduced from \cite[Theorem 11]{Bedford}, which deals with more general integrands and consequently requires $W^{1,\infty}$-regularity of the solution and proves local minimality only with respect to the $W^{1,\infty}$-topology.

We start by recalling the expressions for the first and second variation of the limit energy. A quick derivation is presented for the reader's convenience, although some of the calculations are already present in \cite{MM2} (for the first variation) and a formal derivation of the second variation has been shown in \cite{chouaieb}. For notational convenience, we introduce the function $q_2:(0,L)\times\R^3\to [0,+\infty)$ defined implicitly by $q_2(x_1,\omega_{A})=Q_2(x_1,A)$ and denote by $\nabla_yq_2(x_1,A)$ the derivative of $q_2$ with respect to the second variable. We also recall that the tangent space to the space of admissible deformations $\hat{\mathcal A}$ at a configuration $R$ is given by $\mathcal T R$, where $\mathcal T$ is a space of test functions depending on $\hat{\mathcal A}$. In the unconstrained case when $\hat{\mathcal A}=W^{1,2}((0,L), SO(3))$, it is well-known that $\mathcal T=W^{1,2}((0,L),\mathbb{M}^{3\times 3}_{skew})$. If the additional constraints in Corollary \ref{convbc} are taken into account, the following characterizations hold: 
\begin{itemize}
 \item[(i)] when $\hat {\mathcal A}=\{R \in W^{1,2}((0,L), SO(3)): R(0)=R_0\}$ for a fixed $R_0 \in SO(3)$, then
 $$
 \mathcal T=\{B \in W^{1,2}((0,L),\mathbb{M}^{3\times 3}_{skew}): B(0)=0\}\,;
 $$
 \item[(ii)] when $\hat {\mathcal A}=\{R \in W^{1,2}((0,L), SO(3)): R(0)=R_0 \mbox{ and }R(L)=R_L\}$ for  fixed $R_0$, $R_L \in SO(3)$, then
 $
 \mathcal T=W^{1,2}_0((0,L),\mathbb{M}^{3\times 3}_{skew})
 $;
 \vspace*{5pt}
 \item[(iii)] when $\hat {\mathcal A}=\{R \in W^{1,2}((0,L), SO(3)): R(0)e_1=R(L)e_1=e_1\}$ , then
 $$
 \mathcal T=\{B \in W^{1,2}((0,L),\mathbb{M}^{3\times 3}_{skew}): B(0)e_1=B(L)e_1=0\}\,.
 $$
\end{itemize}
In all these cases, $\mathcal T$ is a linear subspace of $W^{1,2}((0,L),\mathbb{M}^{3\times 3}_{skew})$.

\begin{proposition}\label{variations}
Let $E_0$ be defined as in \eqref{deflimit} and $R\in W^{1,2}((0,L),SO(3))$ be a local minimizer of $E_0$ in $\hat{\mathcal A}$ with respect to the strong $W^{1,2}$-topology. Then $R$ satisfies the following first-order necessary condition for every $B\in \mathcal T$:
\begin{equation}\label{weakform}
\int_0^L\langle\nabla_y q_2(x_1,\omega_{A}(x_1)),R^T(x_1)\omega_{B^{\prime}}(x_1)\rangle-2\langle f,(\omega_{B}(x_1)\times R(x_1)e_1)\rangle\,\mathrm{d}x_1=0\,,
\end{equation}
where $A=R^T\,R^{\prime}$. 
\\
Furthermore, if $R\in \hat{\mathcal A}$ solves \eqref{weakform}, the second variation of $E_0$ at $R$ is given by
\begin{align}\label{secondvariation}
D^2E_0(R)[BR,BR]=
\frac{1}{2}\int_0^L\left[\langle \mathbf{C}R^T\omega_{B^{\prime}},R^T\omega_{B^{\prime}}\rangle-\langle R\nabla_yq_2(x_1,\omega_{A}),B\omega_{B^{\prime}}\rangle-2\langle f,B^2Re_1\rangle\right]\,\mathrm{d}x_1
\end{align}
for every $B\in \mathcal T$, where $\mathbf{C}:=D^2_yq_2(x_1,v)$ is a constant matrix.
\end{proposition}

\begin{proof}
Fix a local minimizer $R\in \hat{\mathcal A}$. Given $B\in \mathcal T$, for $\e$ small there exists a smoothly varying one-parameter family $\gamma_{\e}\in \hat{\mathcal A}$ such that
$\gamma_0=R$ and $\dot{\gamma}_0=BR$, where $\dot{\gamma}$ denotes differentiation with respect to the parameter $\e$. We set $A_{\e}:=\gamma_{\e}^T\gamma^{\prime}_{\e}$ as well as $B_{\e}=\dot{\gamma}_{\e}\gamma_{\e}^T\in\mathcal{T}$, and observe that $A_0=A$ and $B_0=B$. Exchanging the order of differentiation we get 
\begin{equation*}
\gamma_{\e}\dot{A}_{\e}\gamma_{\e}^T=B_{\e}^T\left(\gamma_{\e}(\gamma_{\e}^{\prime})^T\right)^T+\gamma_{\e}\gamma_{\e}^T\dot{\gamma}_{\e}^{\prime}\gamma_{\e}^T
=B_{\e}\gamma_{\e}(\gamma_{\e}^{\prime})^T+\dot{\gamma}_{\e}^{\prime}\gamma_{\e}^T=\dot{\gamma}_{\e}(\gamma_{\e}^{\prime})^T+\dot{\gamma}_{\e}^{\prime}\gamma_{\e}^T=B_{\e}^{\prime},
\end{equation*}
where used that $B_{\e}$ and $\gamma_{\e}\gamma_{\e}^{\prime}$ are skew-symmetric almost everywhere. Using the general equality $\omega_{\tilde{R}^T\tilde{A}\tilde{R}}=\tilde{R}^T\omega_{\tilde{A}}$, valid for all $\tilde{R}\in SO(3)$ and $\tilde{A}\in\mathbb{M}^{3\times 3}_{skew}$, we have
\begin{equation}\label{omegaAeps}
\omega_{\dot{A}_{\e}}=\gamma_{\e}^T\omega_{B_{\e}^{\prime}}\,.
\end{equation}
Due to Proposition \ref{ondensity} (iii) and the estimate (\ref{infbound}) we may indeed differentiate inside the integral in $E_0(\gamma_\e)$ with respect to $\e$. We then obtain
\begin{align*}
&\frac{\mathrm{d}}{\mathrm{d}\e}\frac{1}{2}\int_0^LQ_2(x_1,\gamma_{\e}^T(x_1) \gamma^{\prime}_{\e}(x_1))-2\langle f,\gamma_{\e}(x_1)e_1\rangle\,\mathrm{d}x_1\\
&=\frac{1}{2}\int_0^L\langle\nabla_y q_2(x_1,\omega_{A_{\e}}(x_1)),\omega_{\dot{A}_{\e}}(x_1)\rangle-2\langle f,\dot{\gamma}_{\e}(x_1)e_1\rangle\,\mathrm{d}x_1.
\end{align*}
Evaluating the above expression at $\e=0$ where it must vanish by local minimality and using \eqref{omegaAeps}, we obtain \eqref{weakform}.

For the second part of the statement, first notice that the claim about $\mathbf{C}$ follows from Proposition \ref{ondensity} (iii). Now consider $R\in \hat{\mathcal A}$ (for the moment, not necessarily a minimizer) and let $\gamma_{\e}$ be a one-parameter family as above. Again we may differentiate inside the integral to obtain
\begin{equation}\label{secondder}
\frac{\mathrm{d}^2}{\mathrm{d}\e^2}E_0(\gamma_{\e})= \frac{1}{2}\int_0^L\langle \mathbf{C}\omega_{\dot{A}_{\e}},\omega_{\dot{A}_{\e}}\rangle +\langle \nabla_yq_2(x_1,\omega_{{A}_{\e}}),\omega_{\ddot{A}_{\e}}\rangle-2\langle f,\ddot{\gamma}_{\e}e_1\rangle\,\mathrm{d}x_1.
\end{equation}
By (\ref{omegaAeps}) and a direct computation we have
\begin{equation*}
\omega_{\ddot{A}_{\e}}=\dot{\gamma}_{\e}^T\omega_{B_{\e}^{\prime}}+\gamma_{\e}^T\omega_{\dot{B}_{\e}^{\prime}},\quad\quad\ddot{\gamma}_{\e}=\dot{B}_{\e}\gamma_{\e}-\dot{\gamma}_{\e}\dot{\gamma}_{\e}^T\gamma_{\e},
\end{equation*} 
which implies by definition of $\gamma_{\e}$ that
\begin{align*}
\omega_{\ddot{A}_0}=R^T(\omega_{\dot{B}_0^{\prime}}-B\omega_{B^{\prime}}),\quad \quad \ddot{\gamma}_0=(\dot{B}_0+B^2)R.
\end{align*}
Hence evaluating (\ref{secondder}) at $\e=0$ we infer by the above equalities and \eqref{omegaAeps} that
\begin{align*}
D^2E_0(R)[BR,BR]=&\frac{1}{2}\int_0^L\langle \mathbf{C}R^T\omega_{B^{\prime}},R^T\omega_{B^{\prime}}\rangle-\langle R\nabla_yq_2(x_1,\omega_{A}),B\omega_{B^{\prime}}\rangle-2\langle f,B^2Re_1\rangle\,\mathrm{d}x_1\\
&+\int_0^L\langle\nabla_yq_2(x_1,\omega_{A}),R^T\omega_{\dot{B}_0^{\prime}}\rangle-2\langle f,(\omega_{\dot{B}_0}\times Re_1)\rangle\,\mathrm{d}x_1
\end{align*} 
If $R$ solves now \eqref{weakform}, then the last integral vanishes since $\dot{B}_0\in \mathcal T$. This proves \eqref{secondvariation}.
\end{proof}

\begin{remark}\label{ELbc}
Note that \eqref{weakform} is the weak formulation of the system
\begin{equation}\label{ODE}
\begin{cases}
&\left(R(x_1)\nabla_y q_2(x_1,\omega_{A}(x_1))\right)^{\prime}=2f\times R(x_1)e_1,\\
& R^{\prime}(x_1)=R(x_1)A(x_1),\\
\end{cases}
\end{equation}
endowed with boundary conditions that depend on the choice of $\hat{\mathcal A}$. In the unconstrained case when we only assume \eqref{weakbc}, system \eqref{ODE} has to be coupled with the natural boundary conditions
\begin{align}\label{neumannbc}
\nabla_yq_2(0,\omega_{A}(0))=\nabla_yq_2(L,\omega_{A}(L))=0.
\end{align}
If we consider the clamped-clamped case \eqref{clamp}-\eqref{clamp-clamp}, then \eqref{neumannbc} is simply replaced by the Dirichlet boundary conditions $R(0)=R_0$ and $R(L)=R_L$ for fixed $R_0$ and $R_L \in SO(3)$. \\
In the case of weak clamping \eqref{rigid}, we have $R(0)e_1=R(L)e_1=e_1$ while $\omega_{B}(0)$ and $\omega_{B}(L)$ are parallel to $e_1$ for all test functions $B \in \mathcal T$. We therefore get the mixed boundary conditions
\begin{align}\label{mixedbc}
R(0)e_1=R(L)e_1=e_1\quad \mbox{ and  }\quad\langle\nabla_yq_2(0,\omega_{A}(0)),e_1\rangle=\langle\nabla_yq_2(L,\omega_{A}(L)),e_1\rangle=0\,.
\end{align}
\end{remark}


The proof of the announced local minimality criterion is in our case considerably simplified, since we can extend $E_0$ in a neighborhood of $SO(3)$ to a functional $\tilde E_0 \in W^{1,2}((0,L),\mathbb{M}^{3\times 3})$ (see \eqref{auxfun} below) that proves to be twice continuously  Fr\'{e}chet-differentiable in $ W^{1,2}$. This is crucially due to a very special structure: $\tilde E_0$ is quadratic in the derivative $M'$, while the $L^\infty$-compactness provided by the Sobolev embedding allows us to consider a general smooth dependence on $M$. Notice that $C^2$ Fr\'{e}chet-differentiability  is a very strong property, which in general fails for integral functionals on $W^{1,2}$, also if the integrand is smooth (see \cite[Example 2.3]{Noll}). It is indeed a known fact in the literature that a quadratic structure is needed (see again  \cite[Introduction]{Noll}, which also refers to \cite{Nem-Sem}). Since however we did not find a precise statement for the case we will consider, we prefer to give a self-contained proof in the lemma below.

\begin{lemma}\label{frechetdiff}
Let $U\subset W^{1,2}((0,L),\mathbb{M}^{m\times n})$ be an open set and $F:U\to\R$ be a functional of the type 
$$
F(M)=\int_0^L \phi(t,M(t),M^{\prime}(t))\,\mathrm{d}t\,,
$$
for a Carath\'{e}odory integrand $\phi(t, V, W)$ which is smooth in the last two variables, is quadratic in $W$ and fulfills
\begin{equation}
\begin{split}
&|\partial_V \phi(t,V,W)|+|\partial_V^2\phi(t,V,W)|\leq C(1+|W|^2),\\
&|\partial_{W} \phi(t,V,W)|+|\partial_{V}\partial_W\phi(t,V,W)|\leq C(1+|W|),\\
&\partial_{W}^2\phi(t,V,W)=f(t,V)\end{split}\label{bounds}
\end{equation}
with $f(t,\cdot)$ bounded and continuous uniformly in $t$. Then $F$ is twice continuously Fr\'{e}chet-differentiable on $U$.
\end{lemma}
 
\begin{proof}
It is a standard fact that the assumptions imply $F \in C^1(U)$ with first derivative
\begin{equation*}
DF(M)[h]=\int_0^L\partial_V\phi(t,M(t),M^{\prime}(t))h(t)+\partial_{W}\phi(t,M(t),M^{\prime}(t))h^{\prime}(t)\,\mathrm{d}t.
\end{equation*}
We continue with the second derivative. The natural candidate is given in compact form by
\begin{equation*}
D^2F(M)[h_1,h_2]=\int_0^L\left[D^2\phi(t,M(t),M^{\prime}(t))(h_2(t),h_2^{\prime}(t))\right]:(h_1(t),h_1^{\prime}(t))\,\mathrm{d}t,
\end{equation*}
where the term inside the brackets has to be understood as the product of a tensor and a matrix. Using the bounds (\ref{bounds}) and the Sobolev embedding theorem it follows that $D^2F(M)$ is a continuous bilinear form on $W^{1,2}((0,L),\mathbb{M}^{m\times n})$. Rewriting the difference 
$$A(h):=(DF(M+h_1)-DF(M))[h_{2}]-D^2F(M)[h_1,h_2]$$ as a double integral we infer from Fubini's theorem and H\"older's inequality that
\begin{align*}
&|A(h)|\leq \int_0^1\int_0^L\left|\left[\left(D^2\phi(t,M+sh_1,M^{\prime}+sh_1^{\prime})-D^2\phi(t,M,M^{\prime})\right)(h_2,h_2^{\prime})\right]:(h_1,h_1^{\prime})\right|\,\mathrm{d}t\,\mathrm{d}s
\\
&\leq \|h_1\|_{\infty}\|h_2\|_{\infty}\int_0^1\int_0^L\left|\partial_V^2\phi(t,M+sh_1,M^{\prime}+sh^{\prime}_1)-\partial_V^2\phi(t,M,M^{\prime})\right|\,\mathrm{d}t\,\mathrm{d}s
\\
&+2\max_{i\neq j}\left(\|h_i\|_{\infty}\|h_j^{\prime}\|_{L^2}\right)\int_0^1\left(\int_0^L\left|\partial_V\partial_{W}\phi(t,M+sh_1,M^{\prime}+sh^{\prime}_1)-\partial_V{\partial_{W}}\phi(t,M,M^{\prime})\right|^2\,\mathrm{d}t\right)^{\frac{1}{2}}\mathrm{d}s
\\
&+\|h_1^{\prime}\|_{L^2}\|h_2^{\prime}\|_{L^2}\sup_{t,s}\left|f(t,M(t)+sh_1(t))-f(t,M(t))\right|
\end{align*}
Since we have to estimate the operator norm we can assume that $\|h_2\|_{W^{1,2}}\leq 1$. When $\|h_1\|_{W^{1,2}}\to 0$, the integrals vanish again by the bounds (\ref{bounds}) and the general Lebesgue dominated convergence theorem applied twice. The supremum in the last line vanishes by the assumptions on $f$. Hence $F$ is indeed twice Fr\'{e}chet-differentiable in $M$.

It remains to show that the second derivative is continuous on $U$. To this end take any $h_1,h_2$ such that $\|h_1\|_{W^{1,2}},\|h_2\|_{W^{1,2}}\leq 1$ and suppose that $M_n\to M$ in $W^{1,2}((0,L),\mathbb{M}^{m\times n})$. Then, as above, the difference $\Delta_n(h):=\left(D^2F(M_n)-D^2F(M)\right)[h_1,h_2]$ can be estimated as
\begin{align*}
\left|\Delta_n(h)\right|\leq& \int_0^L\left|\left[\left(D^2\phi(t,M_n,M_n^{\prime})-D^2\phi(t,M,M^{\prime})\right)(h_2,h_2^{\prime})\right]:(h_1,h_1^{\prime})\right|\,\mathrm{d}t\\
\leq& C \int_0^L|\partial_{V}^2\phi(t,M_n,M_n^{\prime})-\partial_V^2\phi(t,M,M^{\prime})|\,\mathrm{d}t+C\sup_{t}|f(t,M_n(t))-f(t,M(t))|
\\
&+ C\left(\int_0^L|\partial_{V}\partial_{W}\phi(t,M_n,M_n^{\prime})-\partial_V\partial_{W}\phi(t,M,M^{\prime})|^2\,\mathrm{d}t\right)^{\frac{1}{2}}.
\end{align*}
Again the right hand side is independent of $(h_1,h_2)$ and vanishes when $n\to\infty$ by the Sobolev embedding theorem and the bounds (\ref{bounds}) combined with the general Lebesgue dominated convergence theorem. Hence we proved the continuity.
\end{proof}

\begin{remark}\label{frechetbc}
The statement of Lemma \ref{frechetdiff} remains valid if $F$ is defined on an open set $U$ with respect to an affine subspace of $W^{1,2}((0,L),\mathbb{M}^{m\times n})$. This will be useful to include boundary conditions.
\end{remark}
Now we can state our main result of this section:

\begin{theorem}\label{mainlocal}
Assume that the functional $E_h$ defined in (\ref{defenergy}) is lower semicontinuous with respect to weak convergence in $W^{1,2}(\Omega,\R^3)$. Moreover let $R=(\partial_1 v|d_2|d_3)\in \hat{\mathcal{A}}$ be a solution of (\ref{weakform}). Assume that $D^2E_0(R)[BR,BR]\geq c\|BR\|^2_{W^{1,2}}$ for some $c>0$ and all $B\in \mathcal{T}$. Then $R$ is a strict local minimizer of $E_0$ with respect to the $L^2$-topology on $\hat{\mathcal{A}}$ and there exists a sequence $v_h$ of local minimizers of $E_h$ in $\mathcal{A}_h$ such that $v_h\to v$ strongly in $W^{1,2}(\Omega,\R^3)$ and $\nabla_h v_h\to R$ strongly in $L^2(\Omega,\mathbb{M}^{3\times 3})$.
\end{theorem}

\begin{proof}
Due to Proposition \ref{weak-strong} and Theorem \ref{conv-loc} it is enough to show that under the assumptions $R$ is a strict local minimum of $E_0$ in $\hat{\mathcal{A}}$ with respect to the $W^{1,2}$-topology.

We prove the statement only in the case of the boundary conditions (\ref{rigid}) since they involve some additional difficulties. We introduce the closed affine space
\begin{eqnarray*}
\mathcal{V}=\Big\{M\in W^{1,2}((0,L),\mathbb{M}^{3\times 3}):&&\;M(0)e_1=M(L)e_1=e_{1},\\ 
 &&\langle M(0)e_i,e_1\rangle =\langle M(L)e_i,e_1\rangle=0,\,i=2,3\Big\}.
\end{eqnarray*}
Note that $\hat{\mathcal{A}}\subset\mathcal{V}$ since for $R\in SO(3)$ the condition $Re_1=e_1$ implies the conditions in the second line. Let us denote by $\Pi_3$ the smooth projection onto $SO(3)$ defined on an open neighborhood $U_0$ of $SO(3)$. We may assume that $\det(M)>0$ for all $M\in U_0$. Setting $2\delta:=\dist(SO(3),\partial U_0)>0$ we introduce the auxiliary functional $\tilde{E}_0:U\subset \mathcal{V}\to\R$ given by
\begin{equation}\label{auxfun}
\tilde{E}_0(M)=E_0(\Pi_3(M))+\frac{1}{2}\|M-\Pi_3(M)\|^2_{W^{1,2}},
\end{equation}
where $U:=\{M\in \mathcal{V}:\;\dist(M(x_1),SO(3))<\delta\}$. Note that by the choice of $\delta$ we can assume that the derivatives up to any finite order of the projection operator $\Pi_3$ are uniformly bounded. In particular, by Lemma \ref{projection} and the chain rule and it holds that $\Pi_3(M)\in\hat{\mathcal{A}}$ for all $M\in U$ with $(\Pi_3(M))^{\prime}=D\Pi_3(M)M^{\prime}$, so that $\tilde{E}_0$ is well-defined and finite. Moreover, as $|(\Pi_3(M))^{\prime}|\leq C|M^{\prime}|$ and
\begin{align*}
\tilde{E}_0(M)=\frac{1}{2}\int_0^LQ_2(\Pi_3(M)^TD\Pi_3(M)M^{\prime})+|M^{\prime}-D\Pi_3(M)M^{\prime}|^2+|M-\Pi_3(M)|^2\,\mathrm{d}x_1,
\end{align*}
using Proposition \ref{ondensity} (iii) one can check that $\tilde{E}_0$ fulfills the assumptions of Lemma \ref{frechetdiff}. By Remark \ref{frechetbc}, it is twice continuously Fr\'{e}chet-differentiable on the open set $U$. Thus it is enough to show that the first derivative vanishes while the second is strictly positive definite in $R$. Using standard variations of the form $R+\e M$ with $M\in \mathcal{V}-R$ we can relate the Fr\'{e}chet-derivative of $\tilde{E}_0$ to the Euler-Lagrange equation of the functional $E_0$. Indeed, since $R$ is a solution of (\ref{weakform}) we obtain
\begin{equation*}
D\tilde{E}_0(R)[M]=\frac{\mathrm{d}}{\mathrm{d}\e}\tilde{E}_0(R+\e M)_{|\e=0}=\frac{\mathrm{d}}{\mathrm{d}\e}E_0(\Pi_3(R+\e M))_{|\e=0}=0.
\end{equation*}
Similarly for the second derivative we deduce
\begin{align*}
D^2\tilde{E}_0(R)[M,M]&=\frac{\mathrm{d}^2}{\mathrm{d}\e^2}\tilde{E}_0(R+\e M)_{|\e=0}=\frac{\mathrm{d}^2}{\mathrm{d}\e^2}E_0(\Pi_3(R+\e M))_{|\e=0}+\|M-(D \Pi_3(R) )M\|^2_{W^{1,2}}\\
&\geq c\|(D \Pi_3(R))M\|^2_{W^{1,2}}+\|M-(D \Pi_3(R))M\|^2_{W^{1,2}}\geq \frac{c}{c+1}\|M\|^2_{W^{1,2}}.
\end{align*}
By standard results of calculus the functional $\tilde{E}_0$ has a strict local minimum at $R$, and so does the constrained functional $E_0$.	
\end{proof}

\begin{remark}\label{necessarycondition}
A necessary condition for $R$ to be a $W^{1,2}$-local minimizer in $\hat{\mathcal{A}}$ is given by $D^2E_0(R)[BR,BR]\geq 0$. This follows for instance by applying \cite[Theorem 11]{Bedford}.	
\end{remark}

\section{Stable configurations in the isotropic case}\label{iso}

In this section we apply the previous results to the case of an isotropic material with a prestrain of the form (\ref{twolayer}) and a cross section fulfilling (\ref{upperlayer}). We will investigate the stability of the straight and the helical configurations as the material is subject to a force $fe_{1}$ applied at the free end $\{L\}\times S$ and to suitable boundary conditions. Using (\ref{abstractdensity}) the limit total energy, now denoted by $E_{0}^{f}$ to highlight the dependence on the force parameter $f$, reads as
\begin{equation}\label{example}
E^{f}_0(R)=\frac{1}{2}\int_0^Lc_{12}a_{12}(x_1)^2+c_{13}(a_{13}(x_1)-k)^2+c_{23}a_{23}(x_1)^2-2f\langle e_1,R(x_1)e_1\rangle\,\mathrm{d}x_1,
\end{equation}
where $A(x_1)=R^T(x_1)R^{\prime}(x_1)$ and $f\in\R$. To reduce notation we set $\mathbf{C}=\text{diag}(c_{23},c_{13},c_{12})$. Note that in this case we have
\begin{equation}\label{vectordensity}
q_2(v)=\langle\mathbf{C}v,v\rangle-2k\langle \mathbf{C}e_2,v\rangle+k^2\langle\mathbf{C}e_2,e_2\rangle.
\end{equation}

\subsection{Stability of the straight configuration}

We start by briefly discussing the boundary conditions under which the straight configuration is a local minimizer for the energy in \eqref{example}. If we only assume the boundary condition (\ref{clamp}) with $R_0=I$, the straight configuration $R(x_{1})=I$ cannot be a local minimizer since the natural boundary conditions in (\ref{ODE}) require $a_{13}(L)=k$. Instead, if we consider the boundary conditions $R(0)e_1=R(L)e_1=e_1$ that correspond to (\ref{rigid}) at both ends, it follows again by Remark \ref{ELbc} that the straight configuration is a critical point of $E_{0}$. In particular such a configuration is a critical point of the energy $E_{0}$ under Dirichlet boundary data corresponding to the clamped-clamped conditions \eqref{clamp} and \eqref{clamp-clamp} with $R_{0}=R_{L}=I$. 

In the following Theorem \ref{idstable} we give sufficient and necessary conditions for local stability of the straight configurations in terms of $f$. We notice that in the weak-clamping case \eqref{rigid} we recover the same critical value of the force which was obtained in \cite{Plos} by means of a formal perturbation argument.

In the statement we will make use of the following auxiliary functions
\begin{align}
\begin{split}
g_1(x)&=\left(x-\frac{c_{12}c_{23}}{(c_{13}kL)^2}x^3\right)\sin(x)+2(\cos(x)-1),\\
g_2(x)&=\left(x-\frac{c_{12}c_{23}}{(c_{13}kL)^2}x^3\right)(\cos(x)+1)-2\sin(x).\end{split}\label{funcg}
\end{align}
and of the following simple lemma.

\begin{lemma}\label{xstar}
Let the functions $g_1,g_2$ be defined in \eqref{funcg} and let $a=\frac{c_{12}c_{23}}{(c_{13}kL)^2}$. If $1-4\pi^{2} a\geq 0$ then the functions do not have any common zero in the interval $(\pi,2\pi)$. If instead $1-4\pi^{2} a< 0$ then there exists a unique common zero $x_{*}$ in the interval $(\pi,2\pi)$.
\end{lemma}
\begin{proof}
First note that $1-4a\pi^2\geq 0$ implies that $x-ax^3\geq 0$ for all $x\in (\pi,2\pi)$, whence $g_2(x)>0$ so that there is no zero in $(\pi,2\pi)$. It remains the case when $1-4\pi^2a<0$. It holds that $g_2(\pi)=0$, $g_2^{\prime}(\pi)=2$ and $g_2(2\pi)=4\pi(1-4\pi^2a)<0$. Therefore there exists $x_*\in(\pi,2\pi)$ such that $g_2(x_*)=0$. Then $(x_*-ax_*^3)=\frac{2\sin(x_*)}{\cos(x_*)+1}$ which yields $g_1(x_*)=0$. We claim that there exists only one zero in $(\pi,2\pi)$. Assume by contradiction that there are at least two zeros. Since $g_1,g_2$ can be extended to holomorphic functions these zeros are isolated. Let $x_1<x_2$ be the two smallest zeros in $(\pi,2\pi)$. Note that $1-ax_{1}^2<0$ and
\begin{align*}
g_1^{\prime}(x)&=g_2(x)-(x-ax^3)+\sin(x)(1-3ax^2),\\
g_2^{\prime}(x)&=-g_1(x)-2+(\cos(x)+1)(1-3ax^2),
\end{align*} 
Hence $g_1^{\prime}(x_1)=-(x_1-ax_1^3)+\sin(x_1)(1-3ax_1^2)>0$ and we conclude that $g_1>0$ for all $x\in(x_1,x_2)$ (as $g_1(x)=0$ implies $g_2(x)=0$ for $x\in(\pi,2\pi)$). It follows that, for all $x\in(x_1,x_2)$, 
\begin{equation*}
g_2^{\prime}(x)\leq -2+(\cos(x)+1)(1-3ax^2)<-2,
\end{equation*}
which is a contradiction.
\end{proof}

\begin{theorem}\label{idstable}
Let $E_{0}^{f}$ be as in \eqref{example} and let $x_{*}$ be as Lemma \ref{xstar}. Then there exists a critical force $f^{crit}$ such that for $f>f^{crit}$ the straight configuration is a strict local minimizer of $E^{f}_{0}$ in the $L^{2}$-topology with the respective boundary conditions.
If instead $f<f^{crit}$ the straight configuration is not a local minimizer. Moreover, it is given by

\begin{eqnarray*}
f^{crit}&=&
\begin{cases}\displaystyle
\begin{cases}\displaystyle
\frac{(c_{13}k)^2}{c_{23}}-\frac{4\pi^2c_{12}}{L^2} &\mbox{if }\displaystyle\frac{(c_{13}kL)^2}{4\pi^2c_{12}c_{23}}\geq 1	
\\ \displaystyle
\max\left\{\frac{(c_{13}k)^2}{c_{23}}-\frac{x_*^2c_{12}}{L^2},-\frac{\pi^2c_{13}}{L^2}\right\} &\mbox{otherwise}\end{cases}
& \hspace{-.4cm}\mbox{for the b.c. $R(0)=R(L)=I$,}
	\\
	\\ \displaystyle
\max\left\{\frac{(c_{13}k)^2}{c_{23}}-\frac{\pi^2c_{12}}{L^2},-\frac{\pi^2 c_{13}}{L^2}\right\}& \hspace{-1.2cm}\mbox{for the b.c. $R(0)e_1=R(L)e_1=e_1$.}
\end{cases}
\end{eqnarray*}	
\end{theorem}

\begin{proof}
	We start observing that, by Proposition \ref{variations} the second variation of $E^f_{0}$ at $I$ is given by the functional
	\begin{equation}\label{secondvarid}
	F^f(\omega)=\int_0^L\langle \mathbf{C}\omega^{\prime},\omega^{\prime}\rangle-c_{13}k(\omega_1\omega_3^{\prime}-\omega_1^{\prime}\omega_3)+f(\omega_2^2+\omega_3^2)\,\mathrm{d}x_1.
	\end{equation}
	Above, to reduce subscripts and referring to formula \eqref{secondvariation}, we have set $\omega=\omega_{B}$ and observed that $\omega'=\omega_{B'}$. \\
	
	We first discuss the case of Dirichlet boundary conditions $R(0)=R(L)=I$. In this case $\omega\in W^{1,2}_0((0,L),\R^{3})$. 
	Thanks to Theorem \ref{mainlocal} and Remark \ref{necessarycondition} we have that
	\begin{equation}\label{critical}
	f^{crit}=\inf\{f\in\R:\,\inf_{\omega\in W^{1,2}_{0}, \|\omega\|=1} F^{f}(\omega)> 0 \}.
	\end{equation}
	Notice that the quadratic structure of $F^{f}$ implies that, for $f>f^{crit}$ $F^{f}(\omega)\geq c\|\omega\|^{2}_{W_{0}^{1,2}}$, so that, by Theorem \ref{mainlocal}, the conclusion holds true. 
It now simply remains to characterize $f^{crit}$ as in the statement of the theorem. \\

	We begin by claiming that $f^{crit}$ can be characterized as the unique $f\in\R$ such that 
	\begin{equation*}
	\inf_{\omega\in W^{1,2}_0, \|\omega\|=1}F^f(\omega)=\min_{\omega\in W^{1,2}_{0}, \|\omega\|=1} F^{f}(\omega)= 0.
	\end{equation*}
	We first show that there exists $\omega\neq 0$ such that $F^{f^{crit}}(\omega)=0$. Indeed, (\ref{critical}) implies that $F^{f^{crit}}\geq 0$. Moreover, if $F^{f^{crit}}(\omega)\geq c\|\omega\|_{W^{1,2}}$ for some $c>0$, then one can use the Poincar\'e inequality to show that $f^{crit}$ violates (\ref{critical}). Hence there exists a sequence $\omega_n\in W^{1,2}_0((0,L),\R^3)$ such that $\|\omega_n\|_{W^{1,2}}=1$ and 
	\begin{equation}\label{infsequence}
	\lim_n\int_0^L\langle \mathbf{C}\omega_n^{\prime},\omega_n^{\prime}\rangle-c_{13}k(\omega_{1,n}\omega_{3,n}^{\prime}-\omega_{1,n}^{\prime}\omega_{3,n})+f^{crit}(\omega_{2,n}^2+\omega_{3,n}^2)\,\mathrm{d}x_1=0.
	\end{equation}
	Up to subsequences, we have that $\omega_n\rightharpoonup\omega$ in $W^{1,2}((0,L))$ and uniformly on $[0,L]$. By lower semicontinuity, the claim follows if $\omega\neq 0$. Suppose $\omega=0$, then passing to the limit in (\ref{infsequence}), we deduce from uniform and weak convergence that also $\omega_n^{\prime}\to 0$ strongly in $L^2((0,L))$ which contradicts the fact that $\|\omega_n\|_{W^{1,2}}=1$. To prove uniqueness of such $f$, take a minimizer $\omega$ with $\|\omega\|=1$. Then it follows that for any $f<f^{crit}$ we have $F^{f}(\w)\leq (f-f^{crit})\int_0^L\omega_2^2+\omega_3^2\,\mathrm{d}x_1<0$ (otherwise $F^{f^{crit}}(\w)>0$). The same argument shows that there exists no $f>f^{crit}$ such that $F^f(\omega)=0$ for some $\omega\neq 0$.

	It is now convenient to decouple the functional $F^{f}(\omega)=F^f_1(\omega_{1},\omega_{3})+F^f_2(\omega_{2})$, where 
	\begin{align*}
	F^{f}_1&(\omega_{1},\omega_{3})=\int_0^Lc_{23}(\omega_1^{\prime})^2+c_{12}(\omega_3^{\prime})^2-c_{13}k(\omega_1\omega_3^{\prime}-\omega_1^{\prime}\omega_3)+f\omega_3^2\,\mathrm{d}x_1,
	\\
	F^{f}_2 &(\omega_{2})=\int_0^Lc_{13}(\omega_2^{\prime})^2+f\omega_2^2\,\mathrm{d}x_1.
	\end{align*}
	By similar arguments used for $F^f(\omega)$ there exist unique $f_1^{crit},f_2^{crit}\in\R$ characterized as those $f_1,f_2\in\R$ such that 
	\begin{equation}\label{f1f2crit}
	\begin{split}
	\inf_{\|(\omega_{1},\omega_{3})\|=1}&F^{f_1}(\omega_1,\omega_3)=\min_{
		\|(\omega_{1},\omega_{3})\|=1} F_{1}^{f_1}(\omega_{1},\omega_{3})=0 , \\ 
	\inf_{\|\omega_2\|=1}&F_2^{f_2}(\omega_2)=\min_{
		\|\omega_{2}\|=1} F_{2}^{f_2}(\omega_{2})=0,
	\end{split}	
	\end{equation}
	where $\|\cdot\|$ denotes $\|\cdot\|_{W_{0}^{1,2}}$.
	We now claim that $f^{crit}=\max\{f_{1}^{crit},f_{2}^{crit}\}$. Indeed, in the case $f_1^{crit}\leq f_2^{crit}$ we take any minimizer $\omega_2\neq 0$ for $F^{f_2^{crit}}$ and define $\omega=(0,\omega_2,0)$. Then $\|\omega\|=1$ and $F^{f_2^{crit}}(\omega)=0$. Moreover, as $F_1^{f_2^{crit}}\geq 0$ it follows that $F^{f_2^{crit}}\geq 0$, thus by uniqueness we get $f_2^{crit}=f^{crit}$. A symmetric argument proves the claim in the remaining case.
	
	Applying the Poincar\'{e} inequality to $\omega_2$, it is immediate to see that $f_{2}^{crit}=-\frac{\pi^2c_{13}}{L^2}$. Hence it remains to calculate $f_{1}^{crit}$. Using integration by parts and the Poincar\'e inequality with sharp constant for $\omega_3$ we obtain
	\begin{equation*}
	F^f_1(\omega)\geq\int_0^Lc_{23}(\omega_1^{\prime})^2+\left(\frac{\pi^2}{L^2}c_{12}+f\right)\omega_3^2+2c_{13}k\omega^{\prime}_1\omega_3\,\mathrm{d}x_1,
	\end{equation*}
	so that if $\left(\frac{\pi^2}{L^2}c_{12}+f\right)c_{23}>(c_{13}k)^2$, then $F^f_1(\omega_{1},\omega_{3})> 0$ for all $\omega_{1},\omega_{3}$ with $\|(\omega_{1},\omega_{3})\|_{W_{0}^{1,2}}=1$. Thus $f_{1}^{crit}\leq \frac{(c_{13}k)^2}{c_{23}}-\frac{\pi^2c_{12}}{L^2}$.

	Now we use the Euler-Lagrange equation to calculate $f_{1}^{crit}$. Note that if $\omega\neq 0$ is such that $F^{f_{1}^{crit}}_1(\omega)=0$, then it is a global minimizer. Hence it solves the following differential equations:
	\begin{equation*}
	\begin{cases}
	c_{23}\omega_1^{\prime\prime}+c_{13}k\omega_3^{\prime}=0,\\
	c_{12}\omega_3^{\prime\prime}-c_{13}k\omega_1^{\prime}-f^{crit}_{1}\omega_3=0.
	\end{cases}
	\end{equation*} 
	The general solutions of the above system are
	\begin{align*}
	\omega_1(x_1)&=a_1+a_2x_1+a_3e^{\lambda x_1}+a_4e^{-\lambda x_1},\\
	\omega_3(x_1)&=b_1+b_2x_1+b_3e^{\lambda x_1}+b_4e^{-\lambda x_1},
	\end{align*}
	where 
	
	\begin{equation}\label{lambdacond}
	\lambda=\sqrt{\frac{f_{1}^{crit}}{c_{12}}-\frac{(c_{13}k)^2}{c_{12}c_{23}}}\quad \mbox{or equivalently } f^{crit}_{1}=c_{12}\lambda^2+\frac{(c_{13}k)^2}{c_{23}}.
	\end{equation}
	
	Note that by the a priori upper bound $f_{1}^{crit}\leq \frac{(c_{13}k)^2}{c_{23}}-\frac{\pi^2c_{12}}{L^2}$ we have $\lambda\in i\R_+$ and $|\lambda|L\geq \pi$. Plugging this ansatz into the equation and taking into account boundary conditions we get a linear system which admits non-trivial solutions under the following zero determinant condition:
	
	\begin{equation*}
	\frac{c_{23}\lambda(e^{\lambda L}-1)}{e^{\lambda L}}\left(-2(c_{13}k)^2(e^{\lambda L}-1)+c_{23}\lambda Lf^{crit}_{1}(e^{\lambda L}+1)\right)=0.
	\end{equation*}
	Using \eqref{lambdacond} in the previous equation we get that either $|\lambda|L\in 2\pi{\mathbb N}$ or
	\begin{equation}\label{false}
	\lambda L\left(\frac{c_{12}c_{23}\lambda^2}{(c_{13}k)^2}+1\right)(e^{\lambda L}+1)-2(e^{\lambda L}-1)=0.
	\end{equation}
	Checking the real and imaginary part of (\ref{false}) we infer that
	\begin{align}\label{systemg}
	&|\lambda|L\left(\frac{c_{12}c_{23}\lambda^2}{(c_{13}k)^2}+1\right)\sin(|\lambda|L)+2(\cos(|\lambda|L)-1)=0,\\
	&|\lambda|L\left(\frac{c_{12}c_{23}\lambda^2}{(c_{13}k)^2}+1\right)(\cos(|\lambda|L)+1)-2\sin(|\lambda|L)=0.\nonumber
	\end{align}
	Using Lemma \ref{xstar} we conclude that 
	\begin{equation*}
	f_{1}^{crit}=\begin{cases}
	\frac{(c_{13}k)^2}{c_{23}}-\frac{4\pi^2c_{12}}{L^2} &\mbox{if $1-4\pi^2\frac{c_{12}c_{23}}{(c_{13}kL)^2}\geq 0$,}\\
	\frac{(c_{13}k)^2}{c_{23}}-\frac{x_*^2c_{12}}{L^2} &\mbox{if $1-4\pi^2\frac{c_{12}c_{23}}{(c_{13}kL)^2}<0$.}
	\end{cases}
	\end{equation*}
	where the first line follows from \eqref{lambdacond} by taking the minimal possible value for $|\lambda|\in {2\pi\over L} {\mathbb N}$ in the case when system \eqref{systemg} has no solutions. By comparing the values of $f_{1}^{crit}$ and $f_{2}^{crit}$ we get the claim. 
	
	We now discuss the case of weak-clamped boundary conditions. In this case the test functions $\omega\in W^{1,2}((0,L),\R^{3})$ are such that $\omega_{2}(0)=\omega_{3}(0)=\omega_{2}(L)=\omega_{3}(L)=0$. Since the second variation is not sensitive to translations in the variable $\omega_1$, we may further suppose that $\omega_{1}(0)=0$. As a result we can follow the same lines as in the previous case, the only difference being the Neumann boundary condition $\omega_{1}'(L)=0$.
	Again we split the functional in two parts and we define the two critical forces $f_1^{crit},f^{crit}_2$ as in \eqref{f1f2crit} for the new space of test functions. The value for $f^{crit}_2$ is the same as in the previous case. The zero determinant condition of the new linear system for critical forces gives $\lambda L\in\pi i\mathbb{Z}$. Since $\lambda=0$ gives only the trivial solution, the critical force $f_{1}^{crit}$ corresponds to $\lambda L=\pi i$ and, by \eqref{lambdacond} (observe that this condition was derived independently of the boundary data) it reads  $f^{crit}_1=\frac{(c_{13}k)^2}{c_{23}}-\frac{\pi^2c_{12}}{L^2}$.
\end{proof}
Using the above stability result, we can conclude that if the force is large enough, then the straight configuration becomes even the unique global minimizer. 

\begin{theorem}
	Under the boundary conditions $R(0)e_{1}=R(L)e_{1}=e_{1}$ or $R(0)=R(L)=I$, there exists $f_{g}$ such that for all $f>f_g$ the straight configuration $R=I$ is the unique global minimizer of $E_0^{f}$.
\end{theorem}

\begin{proof} We prove the statement in the case $R(0)e_{1}=R(L)e_{1}=e_{1}$ which in particular implies the result in the case $R(0)=R(L)=I$.
	We argue by contradiction. Assume that there exists a sequence $f_n\to +\infty$ and an associated sequence $R_n$ such that $E_0^{f_n}(R_n)\leq E_0^{f_n}(I)$. We know from Theorem \ref{idstable} that for $f>f^{crit}$, the straight configuration is a strict local minimum in $L^2((0,L),SO(3))$. As an elementary first step, we note that the $L^{2}$-neighborhood in which $R(x_{1})=I$ is stable can be taken independent of the forces for all $f\geq f^{crit}+1$. This simply follows since increasing the force adds a positive contribution to the energy difference. Hence it is enough to show that $R_n\to I$ in $L^2((0,L))$. Since $\langle(R-I)e_1,e_1\rangle\leq 0$ for every $R\in SO(3)$, we deduce
	\begin{equation}\label{energyinq}
	0\leq E_0^0(R_n)\leq E_0^0(I)+f_n\int_0^L\langle (R_n(x_1)-I)e_1,e_1\rangle\,\mathrm{d}x_1\leq E^0_0(I),
	\end{equation} 
	so that $E_0^0(R_n)$ is bounded. Thus, up to subsequences, we can assume that $R_n\rightharpoonup R_{\infty}$ in $W^{1,2}((0,L))$. Moreover, since $f_n\to +\infty$ and the integral in (\ref{energyinq}) is sign-definite, it follows that $R_ne_1\to e_1$ in $L^1((0,L))$ and thus uniformly by the Sobolev embedding. Passing to the weak limit of $A_n:=R_n^TR_n^{\prime}$, it holds that $A_n\rightharpoonup A_{\infty}=R_{\infty}^TR^{\prime}_{\infty}$ in $L^2((0,L))$. As $R_{\infty}^{\prime}e_1=0$, the matrix-valued function $A_{\infty}$ takes the form
	\begin{equation*}
	A_{\infty}(x_1)=a_{23}(x_1)(e_2\otimes e_3-e_3\otimes e_2).
	\end{equation*}
	Passing to the limit in (\ref{energyinq}) and using weak lower semicontinuity yields $E_0^0(R_{\infty})\leq E_0^0(I)$, which by a direct comparison is only possible for $a_{23}=0$, hence $A_{\infty}=0$ and $R_{\infty}=I$.
\end{proof}

\begin{remark}\label{locnotglob}
	In general there is a gap between the critical forces where the straight configuration becomes a local and a global minimizer. For example, if $kL=2\pi$ and $c_{13}^2<c_{12}c_{23}$, then the straight configuration is a strict local minimizer for $f=0$ according to Theorem \ref{idstable}. Since for such value of the parameters the curved beam defined in (\ref{curvemin}) is also an admissible configuration for the Dirichlet boundary data $R(0)=R(L)=I$, it turns out that this is the unique global minimizer by Remark \ref{curvmin}. Using formula \eqref{quadraticpart} together with \eqref{torsionalrig}, it can be seen that a limit model satisfying the condition $c_{13}^2<c_{12}c_{23}$ can be obtained by taking a rectangular cross section with $w_z<< w_y$.
\end{remark}

We close this section with a bifurcation result. Our aim is to prove that at the critical force another branch of solutions of (\ref{ODE}) emerges. We make use of the Crandall-Rabinowitz-Theorem \cite{CrRa} which reads as follows:

\begin{theorem}[Crandall-Rabinowitz]\label{CR}
Let $X,Y$ be Banach spaces, $V$ a neighbourhood of $(0,\lambda_0)$ in $X\times\R$ and $G:V\to Y$ have the properties

\begin{itemize}
	\item[(i)] $G(t,0)=0$,
	\item[(ii)] the partial derivatives $G_x,G_t,G_{tx}$ exist and are continuous,
	\item[(iii)] $G_x(\lambda_0,0)$ is a Fredholm operator with zero index and $N(G_x(\lambda_0,0))=\text{span}\{v\}$,
	\item[(iv)] $G_{tx}(\lambda_0,0)v\notin R(G_x(\lambda_0,0))$.
\end{itemize}
If $Z$ is any complement of $N(G_x(\lambda_0,0))$ in $X$, then there is a neighbourhood $U$ of $(\lambda_0,0)$ in $\R\times X$, an interval $(-a,a)$ and continuous functions $\varphi:(-a,a)\to\R$ and $\psi:(-a,a)\to Z$ such that $\varphi(0)=0,\,\psi(0)=0$ and
\begin{equation*}
G^{-1}(0)\cap U=\{(\lambda_0+\varphi(t),tv+t\psi(t)):\;|t|<a\}\cup \{(t,0):\;(t,0)\in U\}.
\end{equation*}
If $G\in C^n(V)$, then $\varphi,\psi\in C^{n-1}((-a,a))$.
\end{theorem}

We will use this theorem in order to show that for a generic value of $f^{crit}$ the couple $(f^{crit},I)$ is a bifurcation point.  
We state the result in the general case and prove it for weak-clamping boundary conditions. The proof in the clamped-clamped case, which requires only slight modifications, is left to the reader.

The restriction on the critical forces in the statement below excludes the case of a two-dimensional null space. Note that for non-negative critical forces such a restriction is always satisfied.


\begin{theorem}\label{bifurcation}
Let $E_0^f$ be the functional defined in (\ref{example}) and let $f^{crit}$ be as in Theorem \ref{idstable}. Assume that 
\begin{itemize}
\item[{(i)}] if $R(0)=R(L)=I$ either $\frac{(c_{13}kL)^2}{4\pi^2c_{12}c_{23}}> 1$ or\footnote{We convene that $x_{*}=2\pi$ if $\frac{(c_{13}kL)^2}{4\pi^2c_{12}c_{23}}=1$} $\frac{(c_{13}k)^2}{c_{23}}-\frac{x_*^2c_{12}}{L^2}\neq-\frac{\pi^2c_{13}}{L^2}$, \\
\item[{(ii)}] if $R(0)e_{1}=R(L)e_{1}=e_{1}$ it holds $\frac{(c_{13}k)^2}{c_{23}}-\frac{\pi^2c_{12}}{L^2}\neq-\frac{\pi^2 c_{13}}{L^2}$.
\end{itemize}
Then there exists a sequence $(f_n,R_n)\in \R\times C^{\infty}((0,L),SO(3))$ such that $(f_n,R_n)\to (f^{crit},I)$ in $\R\times C^2((0,L))$ with $R_n\neq I$ being critical points of the energy $E_0^{f_n}$.
\end{theorem}

\begin{proof}
	As in the proof of Theorem \ref{mainlocal}, we let $U\subset\mathbb{M}^{3\times 3}$ be a small neighborhood of $SO(3)$ and let $\Pi_3:U\to SO(3)$ be the smooth projection onto $SO(3)$. We define the vector space 
	\begin{equation*}
	\mathcal{V}^{\prime}:=\{\varphi\in C^2([0,L],\mathbb{M}^{3\times 3}_{skew}):\;\varphi(0)=0,\,\langle\varphi^{\prime}(0)e_3,e_2\rangle=0,\,\varphi(L)e_1=0\}
	\end{equation*}
	and choose an open neighborhood $V^{\prime}$ of $0$ in the $C^2([0,L])$-topology such that $I+\varphi(x)\in U$ for all $\varphi\in V^{\prime}$. To reduce notation we introduce the auxiliary affine function $T:\mathbb{M}^{3\times 3}\to\R^{3}$ defined as $T(M)=-c_{23}m_{23}e_1+c_{13}(m_{13}-k)e_2-c_{12}m_{12}e_3$. We set $G:\R\times V^{\prime}\to C((0,L),\R^3)\times\R$ as
	\begin{equation*}
	G(t,\varphi)=\left(\begin{array}{c}\left(\Pi_3(I+\varphi)T(\Pi_3(I+\varphi)^TD \Pi_3(I+\varphi)\varphi^{\prime})\right)^{\prime}-te_1\times \Pi_3(I+\varphi)e_1\\
	\langle (\Pi_3(I+\varphi)^TD\Pi_3(I+\varphi)\varphi^{\prime})(L)e_3,e_2\rangle\end{array}\right),
	\end{equation*}
	so that due to the definition of $T$ and (\ref{vectordensity}), $G(t,\varphi)=0$ if and only if $R=\Pi_3(I+\varphi)$ solves (\ref{ODE}) and (\ref{mixedbc}) with $f=t$. Indeed, the boundary conditions (\ref{mixedbc}) follow from Lemma \ref{projection}, the second component of $G$ and the chain rule. Note that $G(t,0)=0$ and $G$ has the required regularity of Theorem \ref{CR} (ii). Next we calculate
	\begin{equation}\label{Gphi}
	G_\varphi(f^{crit},0)[h]=\left(\begin{array}{c}-c_{13}kh^{\prime}e_2+T(h^{\prime\prime})+c_{13}ke_2-f^{crit}e_1\times he_1\\
	\langle h^{\prime}(L)e_3,e_2\rangle\end{array}\right),
	\end{equation}
	where we used that $D \Pi_3(I)$ is the projection onto $M^{3\times 3}_{skew}$. We now define $\hat T:\mathcal{V}^{\prime}\to C((0,L),\R^3)\times\R$ as $\hat T(h)=(T(h^{\prime\prime})+c_{13}ke_2,\langle h^{\prime}(L)e_3,e_2\rangle)$ and we notice that this is a Fredholm operator of index zero since it is bijective. Since all the remaining terms in \eqref{Gphi} are a compact perturbation of $\hat T$, it follows that $G_\varphi(f^{crit},0)$ is a Fredholm operator with index zero, too. Moreover, if $G_\varphi(f^{crit},0)[h]=0$, then the axial vector $\omega_{h}$ of $h$ solves the system 
	\begin{equation}\label{bifurcationsystem}
	\begin{split}
	&c_{23}\omega_1^{\prime\prime}+c_{13}k\omega_3^{\prime}=0,\\
	&c_{13}\omega_2^{\prime\prime}-f^{crit}\omega_2=0,\\
	&c_{12}\omega_3^{\prime\prime}-c_{13}k\omega_1^{\prime}-f^{crit}\omega_3=0
	\end{split}
	\end{equation}
	together with the boundary conditions $\omega(0)=0,\,\omega_{2}(L)=\omega_3(L)=0$ and $\omega_1^{\prime}(L)=0$.
	The solutions are given by 
	\begin{equation*}
	\omega^{crit}(x_1)=\begin{pmatrix}
	Lc_{13}k(1-\cos(\frac{\pi}{L}x_1))\\ 0 \\ -\pi c_{23}\sin(\frac{\pi}{L}x_1)
	\end{pmatrix}.
	\end{equation*}
	if $f^{crit}=\frac{(c_{13}k)^2}{c_{23}}-\frac{\pi^2c_{12}}{L^2}$, while in the case $f^{crit}=-\frac{\pi^{2}c_{13}}{L^{2}}$ it is given by $\omega^{crit}(x_1)=(0,\sin(\frac{\pi}{L}x_1),0)$. Since we have assumed that $\frac{(c_{13}k)^2}{c_{23}}-\frac{\pi^2c_{12}}{L^2}\neq-\frac{\pi^2 c_{13}}{L^2}$ the kernel of the system (\ref{bifurcationsystem}) is one-dimensional and $\omega_{h}\in\text{span}\{\omega^{crit}\}$. 
	
	It now remains to check the transversality condition (iv). Note that
	$G_{t\varphi}(f^{crit},0)[h]=(-e_1\times he_1,0)$. We distinguish between two cases: If $f^{crit}=-\frac{\pi^2c_{13}}{L^2}$, then we have to be sure that the following system of differential equations has no solutions:
	\begin{equation*}
	-c_{13}kh^{\prime}e_2+T(h^{\prime\prime})+c_{13}ke_2-f^{crit}e_1\times he_1=\sin\left(\frac{\pi}{L}x_1\right)e_2.
	\end{equation*}
	The second component reads as
	\begin{equation*}
	c_{13}h_{13}^{\prime\prime}+\frac{\pi^2c_{13}}{L^2}h_{13}=\sin\left(\frac{\pi}{L}x_1\right).
	\end{equation*}
	Multiplying the equation with the right hand side and integrating by parts then gives a contradiction as the left hand side vanishes since $h_{13}(0)=h_{13}(L)=0$. In the case $f^{crit}=\frac{(c_{13}k)^2}{c_{23}}-\frac{\pi^2c_{12}}{L^2}$, we obtain the differential equation
	\begin{equation*}
	-c_{13}kh^{\prime}e_2+T(h^{\prime\prime})+c_{13}ke_2-f^{crit}e_1\times he_1=-\pi c_{23}\sin\left(\frac{\pi}{L}x_1\right)e_3
	\end{equation*}
	and the additional condition $h_{23}^{\prime}(L)=0$. Hence, using also the other boundary conditions in $\mathcal{V}^{\prime}$ we can integrate the first component and obtain $c_{23}h_{23}^{\prime}=-c_{13}kh_{12}$,
	which turns the third component into
	\begin{equation*}
	-\frac{\pi^2c_{12}}{L^2}h_{12}-c_{12}h_{12}^{\prime\prime}=-\pi c_{23}\sin\left(\frac{\pi}{L}x_1\right).
	\end{equation*}
	Again multiplying the equation with the right hand side and integrating by parts leads to a contradiction as the left hand side vanishes due to the boundary conditions $h_{12}(0)=h_{12}(L)=0$. 
	
Now we can apply Theorem \ref{CR} to obtain a sequence $(t_n,\varphi_n)$ converging to $(f^{crit},0)$ such that $G(t_n,\varphi_n)=0$. Hence setting $f_{n}=t_{n}$ and $R_{n}=\Pi_3(I+\varphi_n)$ we have that $R_{n}$ are critical points of $E_{0}^{f_{n}}$. Notice that, since $\varphi_n(x)\in\mathbb{M}^{3\times 3}_{skew}$, $\varphi_{n}\neq 0$ and $\mathbb{M}^{3\times 3}_{skew}$ is the tangent space of $SO(3)$ at $I$, it follows that $\Pi_3(I+\varphi_n)\neq I$. The $C^{\infty}$-regularity of $R_n$ follows by a standard bootstrap argument in \eqref{ELbc}.
\end{proof}
 
\subsection{Stability of helical solutions}\label{stabhel}
In this section we apply the results of Section \ref{convisolated} to discuss the stability of helical solutions to \eqref{ELbc} in the clamped-clamped case under their own Dirichlet boundary conditions. By this we mean that we will consider as boundary data $R_{0}$ and $R_{L}$ in \eqref{clamp-clamp} as $R_{0}=R(0)$ and $R_{L}=R(L)$. It is worth noticing that these are the only possible boundary data for proper helices (by this we mean configurations having both non-trivial curvature and torsion) which in general do not satisfy the natural boundary conditions in \eqref{neumannbc} (see Remark \ref{bdproperhelix} below). 

As in the previous section we focus on the isotropic case, thus considering the functional (\ref{example}). 
In terms of the limit variable $R$ we define a helix to be a function of the form
\begin{equation}\label{defhelix}
R(x_1)=R_0e^{Ax_1}Q(\beta x_1+c),
\end{equation}
where $A\in{\mathbb M}^{3\times3}_{skew}$ is a constant matrix, $Q(\theta)=(e_1|\cos(\theta)e_2+\sin(\theta)e_3|\cos(\theta)e_3-\sin(\theta)e_2)$, $R_0\in SO(3),\beta,c\in \R$. The definition is motivated by the fact that under these assumptions the first column of $R(x_{1})$ is the tangent vector of a curve having constant curvature and torsion (a property characterizing helices). In \eqref{defhelix} the matrix $R_{0}$ (which does not affect the energy) is responsible of a constant rotation of the coordinates system while $Q$, through the {\em twist} parameter $\beta$ and the phase shift $c$, accounts for the possibility of rotating the vectors $d_{2},d_{3}$ at constant velocity in the plane orthogonal to the helix. We mention that a first effect of having intrinsic curvature in the functional \eqref{example} is that of ruling out (proper) helical solutions with $\beta\neq 0$. This effect has been already observed in the literature (see the appendix in \cite{tendril}). 
Here for reader's convenience such a feature will be shortly discussed in the proof of Lemma \ref{statiohelix} among other properties. 

In this case it is an easy computation to get that 
\begin{align}\label{omegaR}
\omega_{R^{T}R'}(x_1)=\begin{pmatrix} \beta-a_{23}\\ a_{13}\cos(\beta x_1+c)-a_{12}\sin(\beta x_1+c)\\ -a_{12}\cos(\beta x_1+c)-a_{13}\sin(\beta x_1+c)\end{pmatrix}.
\end{align}
Moreover we remark that the following equivalent formula for \eqref{defhelix} holds by the definition of exponential of a matrix and since $Q\in SO(3)$:
\begin{equation}\label{Qriparameter}
R(x_1)=R_0Q(\beta x_1+c)e^{(Q^{T}(\beta x_1+c)AQ(\beta x_1+c))x_1}
\end{equation}

We have the following proposition we discuss some stationarity conditions for helical solutions.

\begin{proposition}\label{statiohelix}
Assume that an helix $R$ as in (\ref{defhelix}) is a stationary point of the functional $E_0^f$ defined in (\ref{example}) subject to its own boundary conditions at $x_1=0,L$.  Then there exist $A\in{\mathbb M}^{3\times3}_{skew}$ such that $R(x_1)=R(0)e^{Ax_1}$. Furthermore, if additionally $f\neq 0$, then $A R(0)^Te_1=0$.
\end{proposition}

\begin{proof}
Let us consider $R$ as in \eqref{defhelix} and set $\omega=\omega_{R^{T}R'}$ for notational simplicity. First notice, that the first part of the statement is satisfied in the case
$a_{12}=a_{13}=0$, which corresponds to a twisted straight rod. Indeed, by \eqref{omegaR}, in this case the matrix $\tilde{A}:=R^{T}R'$ is constant, and the statement follows upon replacing $A$ with $\tilde A$ and with $R(0)=R_0\,Q(c)$, since the initial and differential condition together uniquely determine $R(x_1)$. We can therefore additionally assume that
\begin{align}\label{addition}
a_{12}^2+a_{13}^2>0\,.
\end{align}
Requiring $R$ to be a stationary point of $E_0^f$ subject to its own boundary conditions amounts to solving the following Euler-Lagrange equation:
\begin{equation}\label{ELhelix}
\left(R(x_1){\mathbf C}(\omega(x_{1})-ke_{2})\right)^{\prime}=fe_1\times R(x_1)e_1.
\end{equation}
On multiplying the differential equation above by $R^T(x_1)$ and taking the first component we have
\begin{equation*}
-c_{13}\omega_{3}(x_1)(\omega_{2}(x_1)-k)+c_{12}\omega_{2}(x_1)\omega_{3}(x_1)=0.
\end{equation*}
This implies that, for all $x_1 \in [0,L]$, it holds either $\omega_{3}(x_1)=0$ or $(c_{12}-c_{13})\omega_2(x_1)=-c_{13}k$. Since $k \neq 0$, it follows that either $\omega_3(x_1)$ or $\omega_2(x_1)$ has to be constant on a non-empty interval: looking at \eqref{omegaR} and \eqref{addition}, this is only possible when $\beta=0$. Then, using \eqref{Qriparameter}, the first part of the statement follows upon replacing $A$ with $Q^{T}(c)AQ(c)$ and for $R(0)=R_0\,Q(c)$.

For the second part of the statement, we begin by observing that, for $R(x_1)=R(0)e^{Ax_1}$, (\ref{ELhelix}) further simplifies to
\begin{equation*}
A(-c_{23}a_{23}e_1+c_{13}(a_{13}-k)e_2-c_{12}a_{12}e_3)=fR^T(x_1)e_1\times e_1.
\end{equation*}
As we assume that $f\neq 0$, differentiating the above equation again yields
\begin{equation}\label{necessary}
e^{-Ax_1}AR(0)^Te_1\times e_1=0.
\end{equation}
For $x_1=0$ we get that there exists $\theta \in \R$ such that $AR(0)^Te_1=\theta e_1$. Therefore we are only left to show that $\theta=0$. If we assume by contradiction that this does not hold, inserting again  into \eqref{necessary} gives $e^{-Ax_1}e_1\times e_1=0$; since $e^{-Ax_1} \in SO(3)$, this implies $e^{-Ax_1}e_1=e_1$ for all $x_1$, and consequently $Ae_1=0$. Since $A$ is skew-symmetric we then have $a_{11}=a_{12}=a_{13}=0$. This entails that the first component of the vector $AR(0)^Te_1$ must be zero, and contradicts the assumption $\theta \neq 0$. The proof is therefore concluded.
\end{proof}
\begin{remark}\label{bdproperhelix}
Let us consider the energy \eqref{example}. For a possibly twisted helix as in \eqref{omegaR} the natural boundary conditions \eqref{neumannbc} corresponding to weak clamping imply that the first component in \eqref{omegaR} must vanish, that is $\beta=a_{23}$. On the other, as shown in the previous proposition, a helix having non-trivial curvature can be stable only if $\beta=0$. Hence only a curved beam can be stable. This justifies our choice of Dirichlet boundary conditions in order to consider proper helices.
\end{remark}

We now derive from Theorem \ref{mainlocal} a sufficient condition for strict local minimality of helices.

\begin{proposition}\label{cpstable}
Let an helix $R(x_1)=R(0)e^{Ax_1}$ be a stationary point of the functional $E_0^f$ in (\ref{example}) subject to its own boundary conditions. Let $\Omega_{A,k}\in{\mathbb M}^{3\time 3}_{skew}$ be the matrix having as axial vector ${\mathbf C}(\omega_{A}-k e_{2})$ and let $r=R(0)^{T}e_{1}$.
Setting 
\begin{equation}\label{DB}
\begin{split}
D&=\frac{1}{2}\Omega_{A,k}-A{\mathbf C},\\
B&=A^T{\mathbf C}A+\frac{1}{2}(\Omega_{A,k}A+A\Omega_{A,k})+fr_1(e_2\otimes e_2+e_3\otimes e_3)-\frac{fr_2}{2}(e_1\otimes e_2+e_2\otimes e_1),
\end{split}
\end{equation}
consider the linear system with constant coefficients
\begin{equation}\label{secondorder}
{\mathbf C}\xi^{\prime\prime}=(D-D^T)\xi^{\prime}+B\xi.
\end{equation}
If for all $t\in(0,L]$ \eqref{secondorder} has only the trivial solution $\xi=0$ within the space $W^{1,2}_0((0,t),\R^3)$, then
$R(x_{1})$ is a strict local minimizer of $E_{0}^{f}$.
\end{proposition}
\begin{proof}
According to Section \ref{convisolated}, the class of admissible test functions for Dirichlet boundary conditions is $W_{0}^{1,2}((0,L),{\mathbf M}^{3\times 3}_{skew})$. For $B$ in such a space we set $\xi=R^T\omega_{B}$ which belongs to $W^{1,2}_0((0,L),\R^3)$. Since the helix $R$ is fixed, one can easily show that there exists a constant (that is uniform whenever $R^{\prime}$ is in a bounded set of $L^2$) such that 
\begin{equation}\label{norms}
\|\xi\|_{W^{1,2}}\geq c\|BR\|_{W^{1,2}}.
\end{equation}
We first consider the case $f\neq 0$. In this case Theorem \ref{statiohelix} implies that $R^T(x_1)e_1=r$ for all $x_{1}\in [0,L]$. By this, the second variation $D^2E^f_0(R)[BR,BR]$ in \eqref{secondvariation} can be written in terms of $\xi$ as 
\begin{equation*}
F(\xi):=\int_0^L \langle {\mathbf C}(\xi^{\prime}+A\xi),\xi^{\prime}+A\xi\rangle+\langle\xi,\Omega_{A,k} (\xi^{\prime}+A\xi)\rangle+f\langle\xi\times r,\xi\times e_1\rangle\,\mathrm{d}x_1.
\end{equation*}
By Theorem \ref{mainlocal} and \eqref{norms}, we simply need to show that $F(\xi)\geq c\|\xi\|^2_{W^{1,2}}$. 
Using that the matrix $B$ is symmetric, one can collect the terms in the above expression to get
\begin{equation*}
F(\xi)=\int_0^L\langle {\mathbf C}\xi^{\prime},\xi^{\prime}\rangle+2\langle\xi,D\xi^{\prime}\rangle+\langle B\xi,\xi\rangle\,\mathrm{d}x_1.
\end{equation*}
Since ${\mathbf C}$ is positive definite, the stability of such a quadratic functional can be studied via the method of conjugate points, that is we search for non-trivial solutions $\xi\in W^{1,2}_0((0,t),\R^3)$ of the equation \eqref{secondorder} where $t\in(0,L]$. It is well-known that if there is no such $t$, then the required estimate holds (see for example \cite{thexprob} \S 6.3, Theorem 6 and the corresponding corollary). This concludes the proof in the case $f\neq 0$.

In the case $f=0$, the proof follows the same line as above. Notice that the property $R(x_1)^{T}e_{1}=r$ may be no longer satisfied, but at the same time $r$ does not intervene in the definition of $B$ since $f=0$.
\end{proof}

The existence of non-trivial solutions to the system \eqref{secondorder} is equivalent to an algebraic condition. In the next proposition we give a precise statement for such a condition.
\begin{proposition}\label{howtosolve}
Let an helix $R(x_1)=R(0)e^{Ax_1}$ be a stationary point of the functional $E_0^f$ in (\ref{example}) subject to its own boundary conditions. For $D$ and $B$ as in \eqref{DB} and $t\in \R$ consider the matrix $M(t)\in{\mathbb M}^{3\times 3}$, defined through
\begin{equation*}
M(t)=\left(\begin{array}{c|c|c}
\xi_1(t) & \xi_2(t) & \xi_3(t)
\end{array}\right),
\end{equation*}
where for $i=1,2,3$ the function $\xi_i(t)$ is the unique solution of \eqref{secondorder} with the initial data $\xi_i(0)=0$ and $\xi_i^\prime(0)=e_i$. If for all $t\in (0,L]$ it holds that $\det(M(t))\neq 0$, then $R(x_{1})$ is a strict local minimizer of $E_{0}^{f}$.
\end{proposition}

\begin{proof}
Since every solution $\xi$ of \eqref{secondorder} satisfying $\xi(0)=0$ can be written as a linear combination of $\xi_1$, $\xi_2$, and $\xi_3$, the condition $\xi(t)=0$ is equivalent to  $\det(M(t))= 0$. The conclusion  follows by Proposition \ref{cpstable}.
\end{proof}

\subsection{Numerical results}

By using the results of the previous section, we here aim to provide some numerical evidence for an experimentally observed behavior of the physical model we are considering (see \cite{Plos}). At the critical force $f^{crit}$ (here assumed to be non zero) provided in Theorem \ref{idstable}, under some condition on the parameters of the problem, helical solutions arbitrarily close at the origin to the straight configuration emerge as a branch of local minima of the energy $E_{0}^{f}$ with respect to their own boundary conditions. Our stability analysis can be related to the results in \cite{Plos} at least for the experiments that are conducted in a quasi-static regime (see \cite[S2.2]{Plossup}).

More precisely we look for helical local minima of $E_{0}^{f}$ under Dirichlet boundary conditions close to the identity matrix. 

We start by proving that for every given $f\neq 0$ there exists a sequence of stationary helical solutions for the energy $E_0^{f}$ converging uniformly to the straight configuration. 

\begin{proposition}\label{flathelix}
Assume that $f\neq 0$. Then there exists a family $R_{\delta}(x_{1}):=R_{\delta}(0)e^{A_{\delta}x_{1}}$ of helical stationary points of $E_{0}^{f}$ with Dirichlet boundary conditions such that $R_{\delta}\to I$ in  $C^{\infty}([0,L],SO(3))$.
\end{proposition}

\begin{proof}
Given $|\delta|<<1$ we consider $R_{\delta}(0)$ such that 
\begin{equation}\label{helicalbc}
r_{\delta}=R_{\delta}(0)^Te_1=(1-\delta)e_1+\sqrt{2\delta-\delta^2}e_2.
\end{equation} 
We look for a stationary point to $E_{0}^{f}$ under the ansatz $\omega_{A_{\delta}}=\theta_{\delta} r_{\delta}$ for some $\theta_{\delta}$ to be suitably chosen.
Such an ansatz is indeed motivated by the necessary condition in Proposition \ref{statiohelix}. We first notice that, under the assumption \eqref{helicalbc}, it holds that $R_{\delta}^T(x_1)e_1=R_{\delta}(0)^Te_1$. Moreover the differential equation (\ref{ELhelix}) becomes an algebraic equation:
\begin{equation}\label{algebraic}
\begin{pmatrix}
a_{12}(c_{13}(a_{13}-k)-c_{12}a_{13})\\
(c_{23}-c_{12})a_{12}a_{23}\\
a_{23}(c_{23}a_{13}-c_{13}(a_{13}-k))\end{pmatrix}=\begin{pmatrix} 0\\ 0\\ -f \sqrt{2\delta-\delta^2} \end{pmatrix}.
\end{equation}
where we have also taken into account that the third component of $r_{\delta}$ is zero by \eqref{helicalbc}. It is readily seen that the only condition on the coefficients fulfilling the first and second equation which are at the same time compatible with $f\neq 0$ in the third one, is $a_{12}=0$. 
By using the explicit expression of $\omega_{A_{\delta}}$ in \eqref{omegaR} (notice that here $\beta=0$ as observed in the proof of Proposition \eqref{statiohelix})
the third equation is satisfied if $\theta_{\delta}$ is a root of the following second order polynomial
\begin{equation*}
p_{\delta}(\theta)=(c_{23}-c_{13})\sqrt{2\delta-\delta^2}(1-\delta)\theta^2+c_{13}k(1-\delta)\theta-f\sqrt{2\delta-\delta^2}.
\end{equation*}
Since $k\neq 0$ the roots of the polynomial are  reals. The case $c_{13}=c_{23}$ leads to $\theta_{\delta}\simeq C\sqrt{\delta}\to 0$ as $\delta\to 0$. If instead  $c_{13}\neq c_{23}$ the former asymptotic behavior is satisfied by the root with smaller absolute value. Using \eqref{omegaR} the helix corresponding to  $\theta_{\delta}$ via the ansatz is such that $a^{\delta}_{13}, a^{\delta}_{23} = O(\theta_{\delta})$. Hence it converges to the straight configuration uniformly on $[0,L]$ as $\delta\to 0$. By the formula $R_{\delta}(x_1)=R_{\delta}(0)e^{A_{\delta}x_1}$ the same convergence holds for all derivatives.
\end{proof}

\begin{remark}\label{comment}
For $f=0$, one can still derive the algebraic system (\ref{algebraic}). As for a proper helix $a_{23}\neq 0$, we obtain $a_{13}=-\frac{c_{13}k}{c_{23}-c_{13}}$ whenever this is well-defined (otherwise there exists no helix). Note that this will not converge to the straight configuration. 

\end{remark}

Equipped with the results of Propositions \ref{howtosolve} and \ref{flathelix}, in what follows we investigate the stability of helical solutions in the vicinity of the straight configuration and $f$ close to $f^{crit}$ by testing numerically the condition $\det(M(t))\neq 0$ for all $t\in (0,L]$. We warn the reader that our numerical tests do not aim at providing a complete description of the physical phenomenon (as they are limited to small ranges of the parameters), but more to provide some insight on the dependence of the model on some of the parameters. 

In what follows we consider the case of a rectangular cross section $S=(-w_y,w_y)\times(-w_z,w_z)$ of fixed area $|S|=1$. We recall that in this case the torsion constant $\tau_{\scriptscriptstyle{S}}$ appearing in formula \eqref{quadraticpart} describing the quadratic part of the energy admits a series representation. Indeed a separation of variables ansatz leads to the following  formula for the torsion function $\varphi$ solving \eqref{torsionpde}:	
\begin{equation*}
\varphi(x_2,x_3)=x_2x_3-\sum_{n\geq 0}Z_n\sin(\zeta_nx_3)\sinh(\zeta_nx_{2}),
\end{equation*}
where $\zeta_n=\frac{(2n+1)\pi}{2w_z}$ and $Z_n=\frac{4(-1)^n}{w_z\zeta_n^3\cosh(\zeta_nw_y)}$. Differentiating inside the sum and using \eqref{torsionconstant} we calculate 
\begin{equation}\label{torsionalrig}
\tau_{\scriptscriptstyle{S}}=16w_yw_z^3\left(\frac{1}{3}-\frac{64w_z}{\pi^5w_y}\sum_{n\geq 0}\frac{\tanh(\zeta_nw_y)}{(2n+1)^5}\right).
\end{equation}
To perform the numerical computations we need to set the following free parameters of the model:
\begin{itemize}
	\item[] $\lambda,\mu$: Lam\'e constants of the material,
	\item[] $L,w_z$: Length of the rod and width of the rectangular cross section,
	\item[] $\chi$: Effective strength of the prestrain,
	\item[] $f$: External force along the $e_1$-direction,
	\item[] $\delta$: parameter in the boundary condition \eqref{helicalbc}.
\end{itemize}
In the case of a rectangular cross section these parameters determine analytically ${\mathbf C}$ and the given intrinsic curvature $k$ (we approximate the torsional rigidity with a suitable partial sum of the series in formula \eqref{torsionalrig}). To consider rubber materials, we choose the Lam\'e constants as $\lambda=0.326$ GPa and $\mu=0.654\cdot10^{-3}$ GPa (p. 78 in \cite{Sadd}) and we divide the energy by $\mu$ to introduce more stable scales. 
The functions $\xi_1(t)$, $\xi_2(t)$, and $\xi_3(t)$ entering in the definition of $M(t)$ can be calculated rewriting \eqref{secondorder} into the equivalent first-order linear system with matrix
\begin{equation}\label{gamma}
\Gamma=\begin{pmatrix} 0 &I\\ {\mathbf C}^{-1}B &{\mathbf C}^{-1}(D-D^T)\end{pmatrix}
\end{equation}

In the following we plot the values of the function $\Delta(t)$, defined as the smallest eigenvalue in modulus of the matrix $M(t)$, for $t\in (0,L]$. We recall that the helix has a strictly positive second variation (in the sense of Theorem \ref{mainlocal}) if and only if $\Delta(t)>0$ for all $t\in (0,L]$.\\
Before discussing the results of our computations, we have to warn the reader that considering a broader range of the parameters than we did involves the solution of some non-trivial numerical issues. For instance, when increasing the value of the prestrain $\chi$, the matrix $\Gamma$ in \eqref{gamma} turns out to have a large positive eigenvalue. Since the solution of \eqref{secondorder} involves combinations of exponential functions, this can heavily affect the accuracy of the computations. While a more complete analysis of the stability would require a delicate treatment for this problem, we decide to confine our plots to ranges of parameters where the issue does not appear, keeping ourself content with highlighting some qualitative behavior of the system.

The next plots single out two factors that can influence the stability of helical solutions, namely aspect ratio $w_{y}/w_{z}$ and intrinsic curvature $k$ (through the prestrain parameter $\chi$). By looking at the presence of conjugate points in the interval $(0,L]$ the following behavior is observed. \\

\noindent {\it Aspect ratio} (see Figure \ref{fig1}): as the aspect ratio $w_{y}/w_{z}=\frac{|S|}{4w_{z}^{2}}=\frac{1}{4w_{z}^{2}}$ increases conjugate points leave the interval $(0,L]$ showing phase transitions from unstable helices to stable helices. On the other hand, stability can be lost again for too big aspect ratios when not compensated by a big enough prestrain (see the case of $w_{z}=0.45$ and $\chi=10,15,20$ in Figure \ref{fig2}). \\

\noindent {\it Intrinsic curvature} (see Figure \ref{fig2}): a bigger prestrain $\chi$, which is to say by formula \eqref{curvature-prestrain} a bigger intrinsic curvature $k$, favors stability of helical solutions. \\
\begin{figure}[h]
\includegraphics[trim=5cm 9cm 2cm 9cm, scale=0.5]{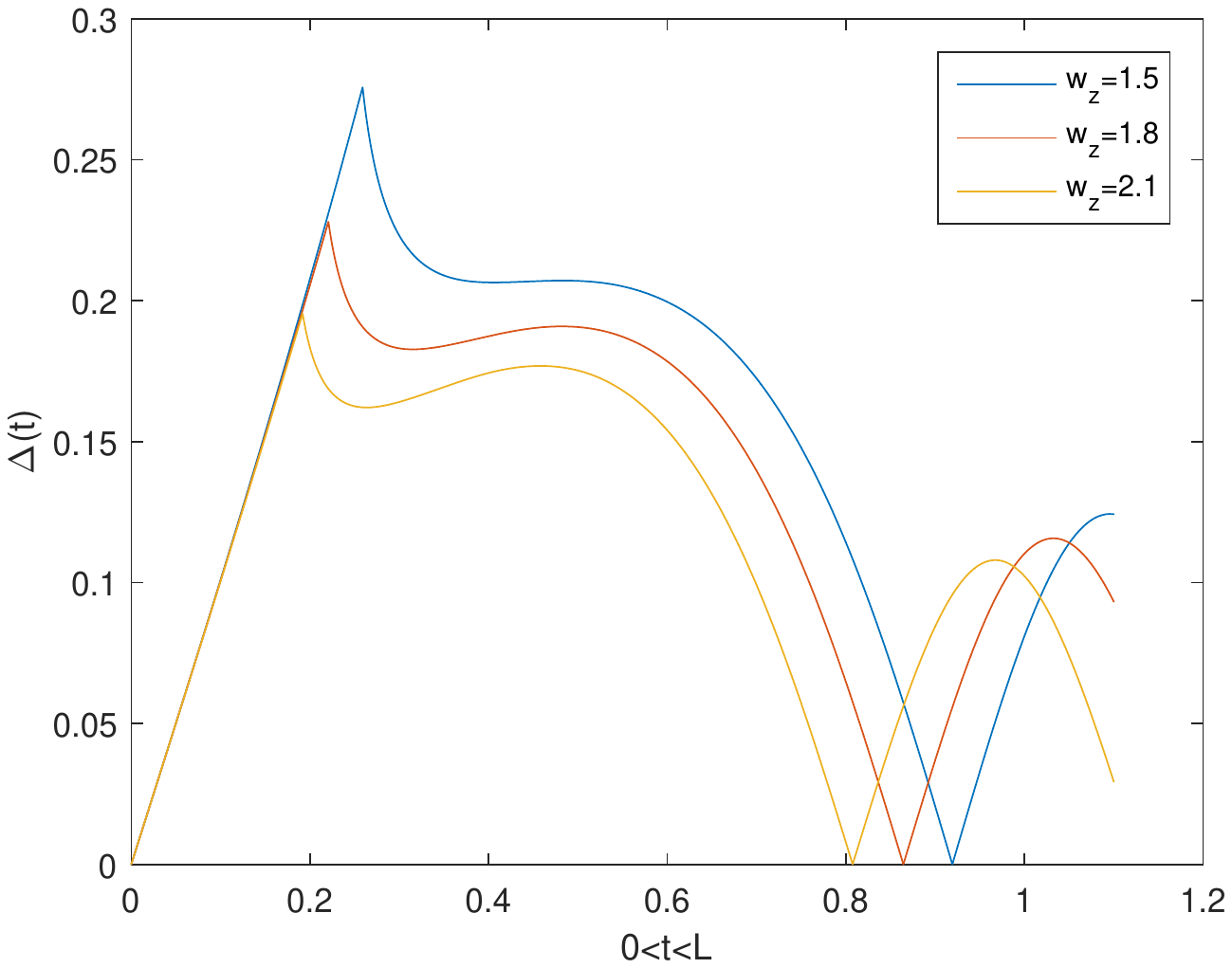}%
\includegraphics[trim=5cm 9cm 0cm 9cm, scale=0.5]{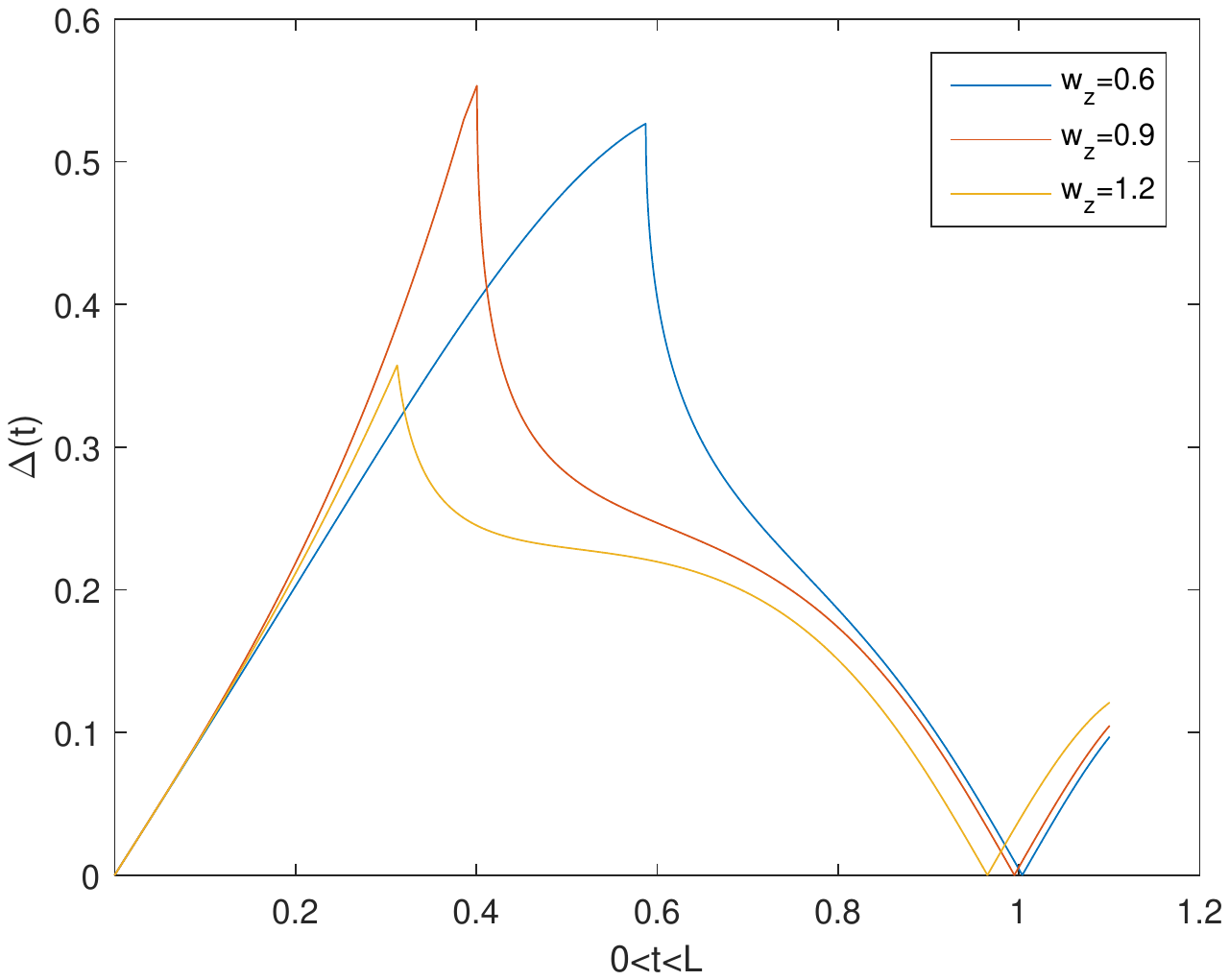}\\
\includegraphics[trim=5cm 9cm 2cm 9cm, scale=0.5]{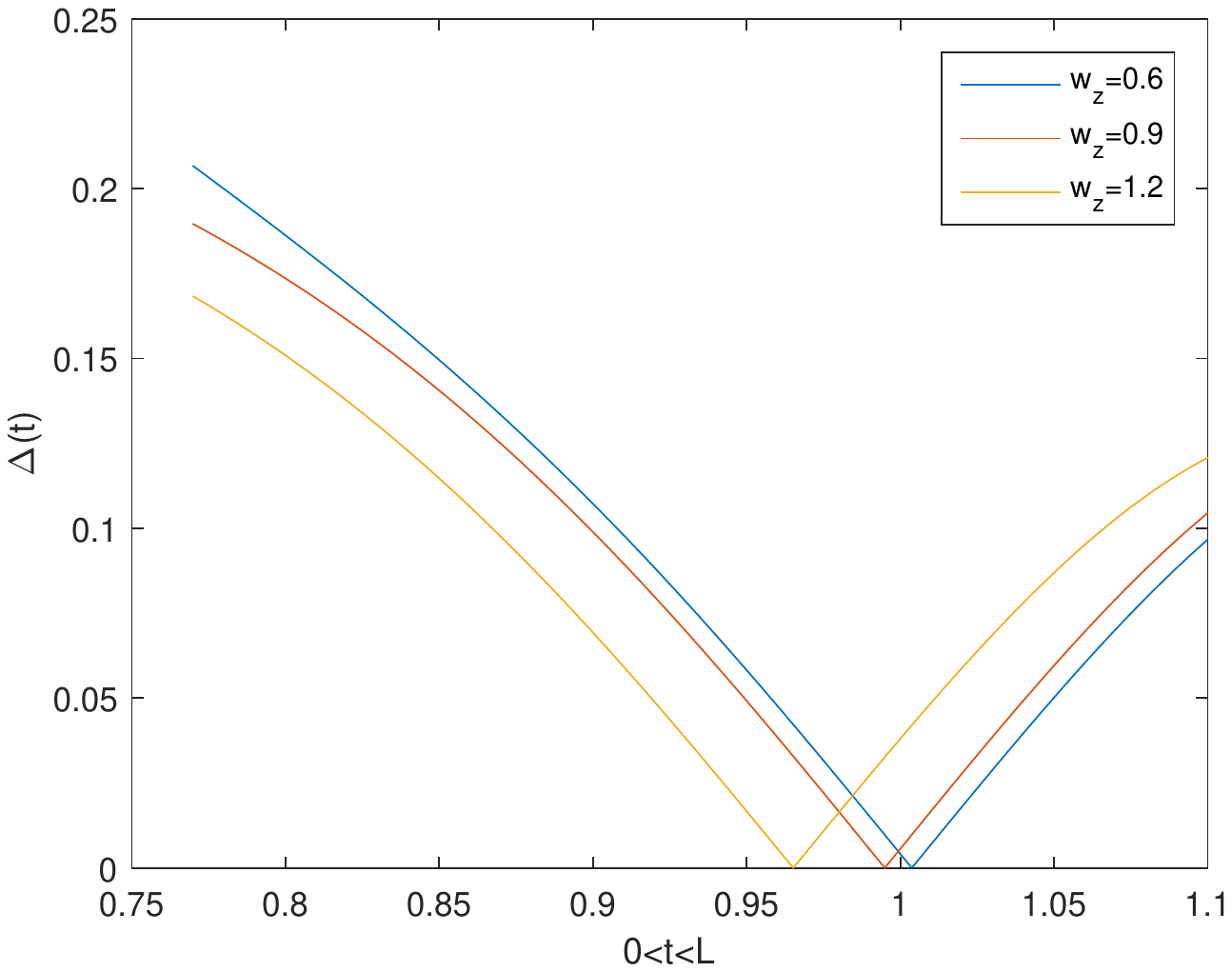}%
\caption{Numerical simulations of $\Delta(t)$ for the choice of parameters $L=1$, $\delta=0.05$, $f=0.999\cdot f^{crit}$, $\chi=6$ and $w_{z}$ decreasing from $2.1$ to $0.6$ (from left to right). As $w_{z}$ decreases (increasing aspect ratio) a conjugate point moves to the right end of the interval $(0,L]$, finally leaving it, thus showing the transition from unstable to stable helices. The bottom picture is a zoomed view of the second one.}
\label{fig1}
\end{figure}

\begin{figure}[h]
\includegraphics[trim=5cm 9cm 0cm 9cm, scale=0.6]{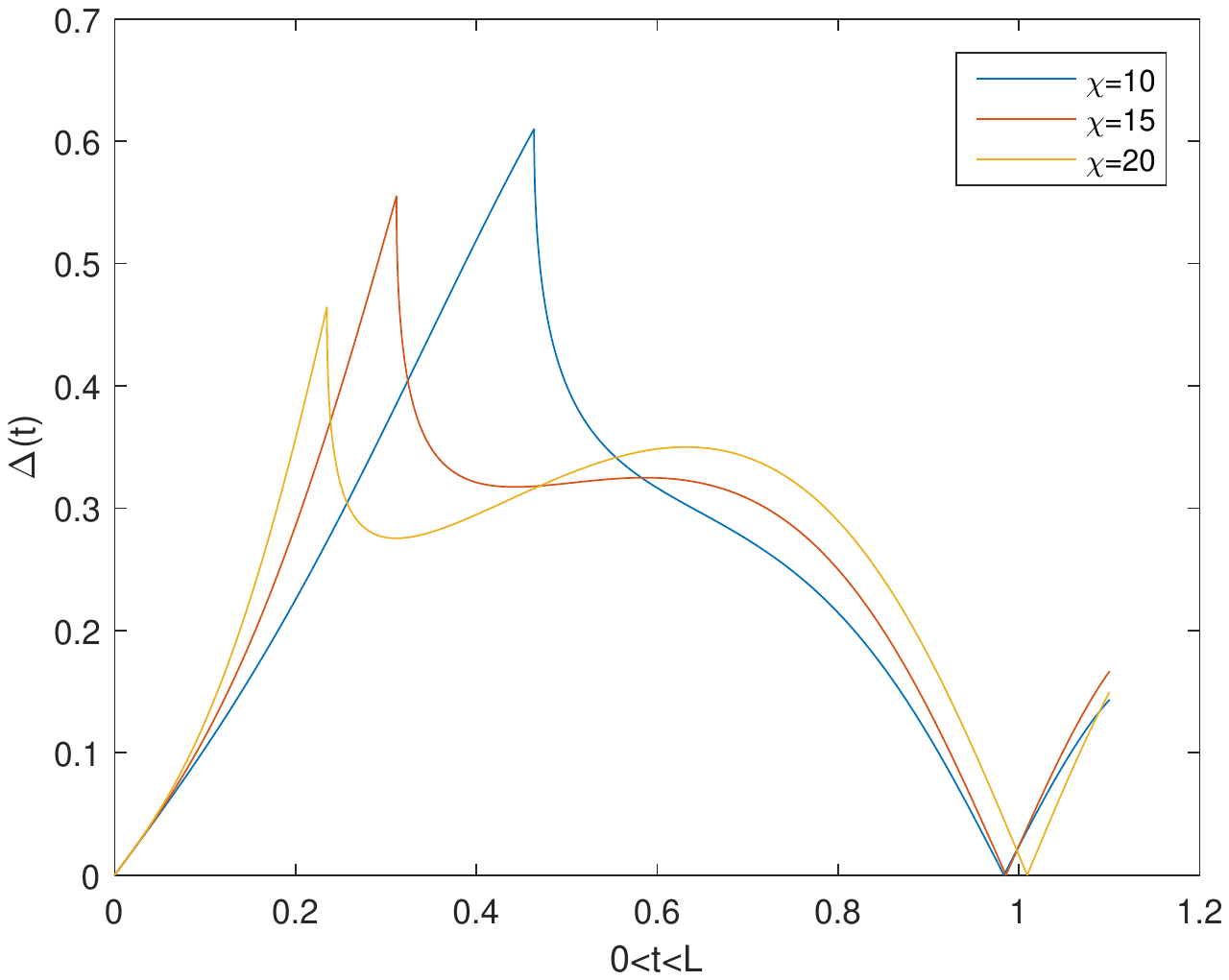}
\captionof{figure}{Numerical simulations of $\Delta(t)$ for the choice of parameters $L=1$, $\delta=0.05$, $f=0.999\cdot f^{crit}$, $w_{z}=0.45$. As the prestrain $\chi$ increases from $10$ to $20$, a conjugate point exits the interval $(0,L]$ leading to the stability of helical solutions.}
\label{fig2}
\end{figure}

\noindent {\bf Acknowledgement} The authors would like to thank G. Albi for some useful discussions on the numerical results contained in the paper.

\appendix
\section{}

Here we prove that the functions $(g,\a)$ can be assumed to be measurable with respect to the product $\sigma$-algebra. 

\begin{lemma}\label{measurable}
Let $A\in L^2((0,L),\mathbb{M}^{3\times 3}_{skew})$ and $\overline{A}\in L^{\infty}(\Omega,\mathbb{M}^{3\times 3})$. Then, for almost every $x_1\in (0,L)$ there exists a minimizer $\alpha_{x_1}$ of (\ref{Q2equiv}). Moreover one can choose $\overline{\a}\in L^2(\Omega,\R^3)$ such that $\partial_k\overline{\a}\in L^2(\Omega,\R^3)$ for $k=2,3$ and $\overline{\alpha}(x_1,\cdot)=\a_{x_1}$.
\end{lemma}
\begin{proof}
Note first that it is convenient to restrict the admissible set to the subspace
\begin{equation*}
\mathcal{M}:=\{\alpha\in W^{1,2}(S,\R^3):\;\int_S\a=0,\,\int_S\partial_2\a_3-\partial_3\a_2=0\},
\end{equation*}
On $\mathcal{M}$ a minimizing pair exists since $Q_3(M)\geq c\,|{\rm sym}(M)|^2$, hence minimizing sequences are weakly compact in $W^{1,2}(S,\R^3)$ by Korn's inequality (note that $S$ is connected). The class $\mathcal{M}$ is closed by this convergence and by weak lower semicontinuity the limit is a minimizer. Since $Q_3$ is strictly convex on symmetric matrices it follows that the minimizer in the set $\mathcal{M}$ is unique.

Testing $\a=0$ in the infimum problem and using the fact that $Q_3(M)\leq C|M|^2$, we deduce that 
\begin{equation}\label{infbound}
Q_2(x_1,A(x_1))\leq C(|A(x_1)|^2+\|\overline{A}\|_{\infty}^2),
\end{equation}
while the lower bound $Q_3(M)\geq c|{\rm sym}(M)|^2$ together with the inequality $|a+b|^2\geq \frac{1}{2}|a|^2-|b|^2$ yields
\begin{equation}\label{infbound2}
Q_2(x_1,A(x_1))\geq c\|\nabla\a\|^2-\frac{1}{c}(|A(x_1)|^2+\|\overline{A}\|_{\infty}^2),
\end{equation}
where we have again used Korn's inequality for a minimizer $(\a_2,\a_3)$. Combining (\ref{infbound}) and (\ref{infbound2}) and the Poincar\'e inequality we find that
\begin{equation}\label{control}
\|\a\|_{W^{1,2}}^2\leq C(|A(x_1)|^2+\|\overline{A}\|_{\infty}^2).
\end{equation}
We now prove joint measurability. For the moment assume that $A$ and $\overline{A}$ are both piecewise constant (to be more precise we consider the case where $A$ is constant on intervals of the form $k/n+(0,1/n)$ with $k\in\mathbb{Z}$ and $\overline{A}$ is constant on cells of the form $z/n+(0,1/n)^3$ with $z\in\mathbb{Z}^3$). Then the minimizer in $\mathcal{M}$ constructed for every $x_1\in (0,L)$ will be piecewise constant in $x_1$, too. Hence in this case $\alpha$ and the weak derivatives are indeed jointly measurable. Integrating (\ref{control}) we obtain that $\a\in L^2(\Omega)$. Notice also that the pointwise obtained weak derivative of $\a(x_1,\cdot)\in W^{1,2}(S,\R^3)$ is also the weak derivative of $\a\in L^2(\Omega,\R^3)$ by Fubini's theorem, hence $\partial_k\a\in L^2(\Omega,\R^3)$ for $k=2,3$.

We conclude with an approximation argument. Let $A_n\in L^2((0,L),\mathbb{M}^{3\times 3}_{skew})$ and $\overline{A}_n\in L^{\infty}(\Omega,\mathbb{R}^{3\times 3})$ be piecewise constant functions converging to $A$ and $\overline{A}$ respectively in the $L^2$-norm and pointwise almost everywhere. We can assume that $\|\overline{A}_n\|_{\infty}\leq\|\overline{A}\|_{\infty}$. For each $n\in\mathbb{N}$ and every $x_1\in(0,L)$ let $\a_n(x_1,\cdot)\in\mathcal{M}$ be the solution defining $Q_2(x_1,A_n(x_1))$. By (\ref{control}) and the Sobolev embedding theorem we may assume that $\a_n(x_1,\cdot)\to \a(x_1,\cdot)$ in $L^2(S)$ for almost every $x_1\in (0,L)$. Note that indeed the whole sequence converges. This follows from the fact that for almost every $x_1\in (0,L)$ the functionals
\begin{equation*}
F_n(\a):=
\begin{cases}
\int_SQ_3\left(\left(\begin{array}{c|c|c} A_n(x_1)(x_2e_2+x_3e_3)+g_n(x_1) &\partial_2\a&\partial_3\a\end{array}\right)+\overline{A}_n(x)\right)\,\mathrm{d}x &\mbox{if $\a\in\mathcal{M}$,}\\
+\infty &\mbox{otherwise}
\end{cases}
\end{equation*}
$\Gamma$-converge with respect to the strong topology on $L^2(S,\R^3)$ to the functional 
\begin{equation*}
F(\a):=
\begin{cases}
\int_SQ_3\left(\left(\begin{array}{c|c|c} A(x_1)(x_2e_2+x_3e_3)+ g(x_1) &\partial_2\a&\partial_3\a\end{array}\right)+\overline{A}(x)\right)\,\mathrm{d}x &\mbox{if $\a\in\mathcal{M}$,}\\
+\infty &\mbox{otherwise.}
\end{cases}
\end{equation*}
To conclude the measurability of $\a$ note that due to (\ref{control}), up to subsequences we have $\a_n\rightharpoonup \overline{\a}$ in $L^2(\Omega,\R^3)$. Then there exists a convex combination of the $\a_n$ that converges strongly to $\overline{\a}$ in $L^2(\Omega,\R^3)$. From Fubini's theorem we deduce that for almost every $x_1\in (0,L)$ it holds that $\overline{\a}(x_1,\cdot)=\a(x_1,\cdot)$. The measurability of the partial derivatives for $k=2,3$ can be proven the same way.
\end{proof}

\begin{remark}\label{ext}
If we consider a bounded extension of $A$ to $\R$ and $\overline{A}$ to $\R\times S$ and then extend the solutions obtained for the piecewise constant approximations in the proof of Lemma \ref{measurable} for fixed $x_1\in \R$ such that $\a_n(x_1,\cdot)\in W^{1,2}(\R^2,\R^3)$ in such a way that (\ref{control}) still holds, then we can prove that $\overline\a\in L^2(\Omega^{\prime},\R^3)$ and $\partial_k\overline\a\in L^2(\Omega^{\prime},\R^3)$ for $k=2,3$ with $\Omega\subset\subset\Omega^{\prime}$.
\end{remark}
Here below we prove a compatibility property of the projection on $SO(3)$ with the weak clamping boundary conditions.

\begin{lemma}\label{projection}
Assume that $M\in\mathbb{M}^{3\times 3}$ is such that $\det(M)>0$ and $Me_1=e_1$ as well as $\langle Me_2,e_1\rangle=\langle Me_3,e_1\rangle=0$. Then $\Pi_3(M)e_1=e_1$.
\end{lemma}
\begin{proof}
Let $M=UP$ be the unique polar decomposition of $M$. It is well known that $\dist(M,O(3))=|U-M|$ (see \cite{NH}). As $\det(M)>0$ we have $U\in SO(3)$, hence $\Pi_3(M)=U$. Setting $M^{11}\in\mathbb{M}^{2\times 2}$ as the submatrix of $M$ where the first column and row have been removed, we have $\det(M^{11})>0$ and thus there exists also a unique polar decomposition $M^{11}=\tilde{U}\tilde{P}$ with $\tilde{U}\in SO(2)$. It is now easy to verify that 
\begin{equation*}
M=\begin{pmatrix} 1&0&0\\0& \tilde{u}_{11}& \tilde{u}_{12}\\
0 &\tilde{u}_{21} &\tilde{u}_{22}\end{pmatrix}\begin{pmatrix}
 1&0&0\\0& \tilde{p}_{11}& \tilde{p}_{12}\\
 0 &\tilde{p}_{21} &\tilde{p}_{22}
\end{pmatrix},
\end{equation*}
hence by uniqueness $Ue_1=e_1$.
\end{proof}

\end{document}